\documentclass[10pt]{gen-j-l}
\usepackage{etex}

\usepackage[utf8]{inputenc}
\usepackage[T1]{fontenc}

\usepackage{amsmath}
\usepackage{amsthm}
\usepackage{amsfonts}
\usepackage{amssymb}
\usepackage{amsrefs}
\usepackage{pstricks}
\usepackage{pst-plot}
\usepackage{eso-pic}
\usepackage{graphicx}
\usepackage{tikz}
\usepackage[cmtip,all]{xy}

\newtheorem{theo}{Theorem}
\newtheorem{lemm}[theo]{Lemma}
\newtheorem{theoa}{Theorem}
\newtheorem{prop}[theo]{Proposition}

\newtheorem{defi}[theo]{Definition}
\newtheorem{example}[theo]{Example}
\newtheorem{rema}[theo]{Remark}

\newcommand{\PP}{\mathbb{P}}

\newcommand{\ZZ}{\mathbb{Z}}

\newcommand{\CC}{\mathbb{C}}

\newcommand{\FF}{{\mathbb{F}}}

\newcommand{\Fb}{\mathbf{F}}
\newcommand{\Cr}{\mathcal{C}}
\newcommand{\Qr}{\mathcal{Q}}

\newcommand{\Ar}{\mathcal{A}}

\newcommand{\Tr}{\mathrm{Tr}}

\newcommand{\dr}{\partial}
\newcommand{\Ric}{\mathfrak{R}}
\newcommand{\dd}{\mathrm{d}}

\newcommand{\card}{\mathrm{Card}}

\newcommand{\MCG}{\mathrm{MCG}}
\newcommand{\PMCG}{\mathrm{PMCG}}
\newcommand{\CV}{\mathrm{Char}}

\title[Isomonodromic deformations arising from quintic curves]{Algebraic isomonodromic deformations of the five punctured sphere arising from quintic plane curves}

\author{Arnaud Girand}
\address{Collège Jean Mermoz --- 24, rue du 2e régiment de dragons --- 02000 Laon (France)}
\curraddr{}
\email{arnaud.girand@ens-cachan.org \text{\textbf{(preferred means of communication)}}}
\thanks{}

\subjclass{14E22, 20G05, 20G20, 32D15, 32G08, 32G34, 34M50, 34M56, 51N15, 55R10}

\date{}

\dedicatory{}

\begin{document}

\begin{abstract}
In this paper, we classify the algebraic isomonodromic deformations that can be obtained through restriction to generic lines of logarithmic flat connections on the complex projective plane $\PP^2_\CC$ whose singular locus is a quintic curve. We then explicitly compute the (finite) finite mapping class group orbits of the associated points in the character variety and describe the new algebraic Garnier solutions that can be obtained through this procedure.
\vspace{1cm}
\textbf{Modified : \today .}
\end{abstract}

\maketitle

\setcounter{tocdepth}{1}
\tableofcontents

\section{Introduction}

\subsection{Representation theoretic result}

The object of this paper is to classify all new Garnier solutions that can be obtained through restriction to generic lines of logarithmic flat connections on the complex projective plane $\PP^2_\CC$ whose singular locus is a quintic curve. In order to do so, we first look at this problem from a purely representation theoretic point. Indeed, if $\nabla$ is some logarithmic flat $\mathfrak{sl}_2(\CC)$--connection over $\PP^2_\CC$ whose polar locus is a quintic curve $Q$, then it is determined, as per the classical Riemann--Hilbert correspondance, by the data of its monodromy representation 
\begin{displaymath}
\rho_\nabla : \pi_1(\PP^2 \setminus Q) \rightarrow SL_2(\CC) \; .
\end{displaymath}

Therefore, our first move will be to classify the group representations $$\rho : \Gamma \rightarrow PSL_2(\CC)$$, where $\Gamma$ is the fundamental group of the complement of some quintic curve in $\PP^2(\CC)$, that satisfy the following conditions:
\begin{enumerate}
\item[\textbf{(C1})] the image of $\rho$ is irreducible and infinite;
\item[\textbf{(C2})] $\rho$ does not factor through a curve (see Definition~\ref{def:factorCurve}).
\end{enumerate}
These two conditions appear, in light of previous work by Diarra~\cite{Diar1} and Mazzocco~\cite{Mazz} to constitute a reasonable starting point to try and produce new algebraic isomonodromic deformations of the five punctured sphere. 

If a monodromy representation $\rho$ satisfies these conditions, then the local monodromy (see Definition \ref{def:locMon}) around any irreducible component of $Q$ must be non--trivial. Indeed, one would otherwise get a factorisation of $\rho$ through the fundamental group of the complement of some curve in $\PP^2$ with degree at most $4$ in which case earlier work by Cousin (see Section 5.1 in~\cite{theseGael}) proves that $\rho$ cannot satisfy both \textbf{(C1)} and \textbf{(C2)}. We prove the following result.

\begin{theoa}\label{thA:quintics}
Let $\Gamma$ be the fundamental group of the complement of some quintic curve $Q$ in $\PP^2(\CC)$ and let $\rho : \Gamma \rightarrow PSL_2(\CC)$ satisfy the following conditions:
\begin{enumerate}
\item[\textbf{(C1})] the image of $\rho$ is irreducible and infinite;
\item[\textbf{(C2})] $\rho$ does not factor through a curve.
\end{enumerate}

Then the triple $(\Gamma,\rho,Q)$ is (up to global conjugacy for $\rho$ and $PGL_3(\CC)$ action for $Q$) one of the following:
\begin{enumerate}
\item $\Gamma \cong \langle a,b,c \, | \, (ab)^2 (ba)^{-2} = (ac)^2 (ca)^{-2}= [b,c]=1 \rangle \;  $,
$$\rho : a  \mapsto \begin{pmatrix}
0 & 1 \\ 
-1 & 0
\end{pmatrix}  , \quad 
b  \mapsto \begin{pmatrix}
u & 0 \\ 
0 & u^{-1}
\end{pmatrix}, \quad
c  \mapsto \begin{pmatrix}
v & 0 \\ 
0 & v^{-1}
\end{pmatrix} \; , \text{ for some }  u,v \in \CC^* \; ,$$
and the curve $Q$ is composed of three lines tangent to a conic ; 
\item $\Gamma \cong \langle a,b,c \, | \, [a,b] = [a,c^{-1}bc] = 1, \, (bc)^2 = (cb)^2\rangle$,
\begin{displaymath}
\rho : a  \mapsto \begin{pmatrix}
u & 0 \\ 
0 & u^{-1}
\end{pmatrix}  , \quad 
b  \mapsto \begin{pmatrix}
v & 0 \\ 
0 & v^{-1}
\end{pmatrix}, \quad
c  \mapsto \begin{pmatrix}
0 & 1 \\ 
-1 & 0
\end{pmatrix} \; , \text{ for some }  u,v \in \CC^* \; ,
\end{displaymath}
and $Q$ is made of three concurrent lines and a conic tangent to two of the latter ;
\item $\Gamma \cong \langle a,b,c \, |, [a,b]=[b, c^2]=1,\, ca = bc\rangle$,
\begin{displaymath}
\rho : a  \mapsto \begin{pmatrix}
t &  0 \\ 
0 & t^{-1}
\end{pmatrix}  , \quad 
b  \mapsto \begin{pmatrix}
t^{-1} & 0 \\ 
0 & t
\end{pmatrix}, \quad
c  \mapsto \begin{pmatrix}
0 & 1 \\ 
-1 & 0
\end{pmatrix} \; , \text{ for some }  t \in \CC^* \; .
\end{displaymath}
Here, the curve $Q$ is one of two special configurations of two lines and a cubic, described in Section~\ref{sec:cubic}. 
\end{enumerate}
Note that the above triples do occur but do not necessarily satisfy conditions \textbf{(C1)} and \textbf{(C2)}, depending on the parameters.
\end{theoa}

\begin{figure}[!h]
\begin{center}
\scalebox{0.5}{
\begingroup%
  \makeatletter%
  \providecommand\color[2][]{%
    \errmessage{(Inkscape) Color is used for the text in Inkscape, but the package 'color.sty' is not loaded}%
    \renewcommand\color[2][]{}%
  }%
  \providecommand\transparent[1]{%
    \errmessage{(Inkscape) Transparency is used (non-zero) for the text in Inkscape, but the package 'transparent.sty' is not loaded}%
    \renewcommand\transparent[1]{}%
  }%
  \providecommand\rotatebox[2]{#2}%
  \ifx\svgwidth\undefined%
    \setlength{\unitlength}{841.68bp}%
    \ifx\svgscale\undefined%
      \relax%
    \else%
      \setlength{\unitlength}{\unitlength * \real{\svgscale}}%
    \fi%
  \else%
    \setlength{\unitlength}{\svgwidth}%
  \fi%
  \global\let\svgwidth\undefined%
  \global\let\svgscale\undefined%
  \makeatother%
  \begin{picture}(1,0.70658683)%
    \put(0,0){\includegraphics[width=\unitlength]{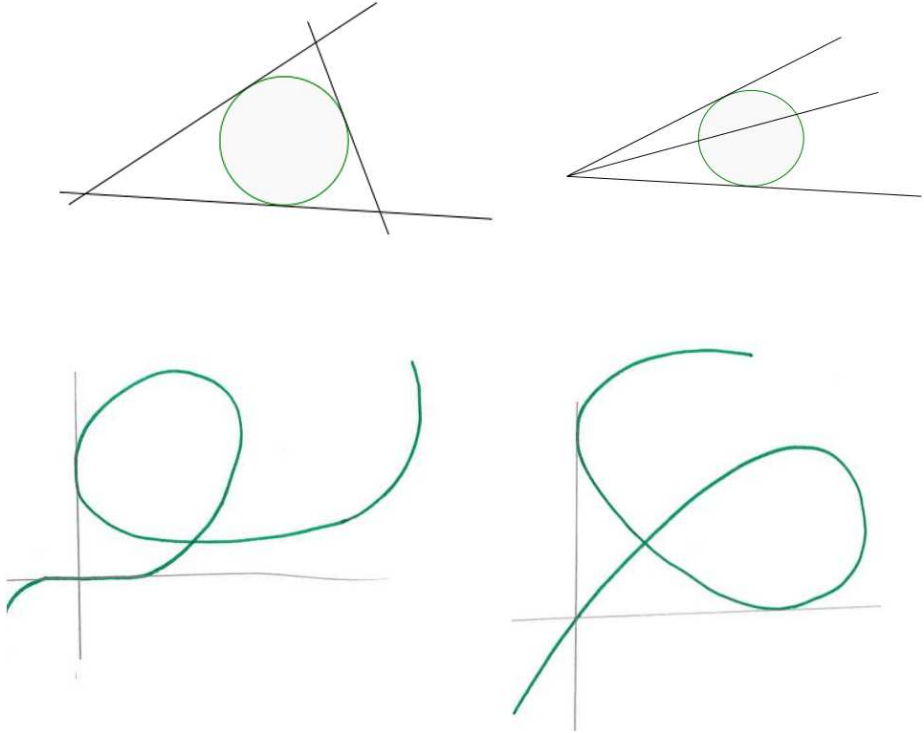}}%
  \end{picture}%
\endgroup}
\end{center}
\caption{Quintic curves appearing in Theorem~\ref{thA:quintics}. On the top row (from left to right) are case 1 and 2, case 3 appears in the bottom row.} \label{fig:quintics}
\end{figure}

The proof of the above theorem, as detailed in Section \ref{sec:Quintics}, is built upon Degtyarev's work on classifying the non--abelian fundamental groups of the complement of quintic plane curves in $\PP^2_\CC$~\cite{Degtyarev}. We give a detailed description of the curves corresponding to each of the cases in Theorem~\ref{thA:quintics}. 
\subsection{Garnier solutions and mapping class group orbits}

In the second part of this paper, we set up the basis of the method we will use to actually construct algebraic isomonodromic deformations corresponding to the group representations appearing in Theorem~\ref{thA:quintics}. Taking inspiration from Hitchin's works on the Painlevé VI equation~\cite{Hitchin} sur l'équation de Painlevé VI, we make the following remark: if one has a logarithmic flat connection $\nabla$ over the projective plane $\PP^2_\CC$ whose polar locus is exactly some degree five algebraic curve, then since any generic line intersects such a curve at five points the family of connections given by the restriction of $\nabla$ to such lines yields an isomonodromic deformation of the five punctured sphere. Moreover, the monodromy representation of any such restricted connection has the same image as that of $\nabla$.

More precisely, let $L$ be some line generically chosen in the complex projective plane $\PP^2_\CC$; then $L$ intersects the polar locus of the connection $\nabla$ at exactly five points, which we can (up to Möbius transformation) assume to be equal to $0, \, 1,$ $\infty$ and $t_1, t_2 \in \CC^*\setminus \lbrace 1 \rbrace$. Restricting $\nabla$ to $L$ one gets a logarithmic flat connection over the five punctured Riemann sphere $\PP^1_5~:= \PP^1_\CC \setminus \lbrace 0,1,t_1,t_2,\infty \rbrace$ whose monodromy representation $\rho_L$ is given by the diagram
\begin{displaymath}
\xymatrix{
\pi_1(\PP^1_5) \cong \Fb_4  \ar[rd]_{\rho_L}  \ar[r]^{\tau}& \pi_1(\PP^2_\CC - Q) \ar[d]^{\rho_\nabla}\\
 & SL_2(\CC) 
}
\end{displaymath}
where $\tau$ is the natural surjective morphism give by the Lefschetz hyperplane theorem (see Theorem 7.4 in~\cite{Milnor}). Since $\nabla$ is flat, it is known~\cite{Malgrange} that $\rho_L$ does not depend (up to conjugacy) on the choice of the generic line $L$. Therefore, there exists a Zariski--open set in the dual $\widehat{\PP^2(\CC)}$ such that all connections in the family $(\nabla_L)_{L \in U}$ have the same monodromy (up to conjugacy).

Using the explicit description of the curves appearing in Theorem~\ref{thA:quintics}, we are then able to actually compute the corresponding monodromy representations. Namely, if one sets $\Fb_4 := <d_1, d_2,d_3,d_4 \, | \, \emptyset\rangle$ then they are (in the same order and up to conjugacy): 
 \begin{align*}
\rho_1~: d_1  \mapsto \begin{pmatrix}
v& 0 \\ 
0 & v^{-1}
\end{pmatrix}\quad  d_2 \mapsto \begin{pmatrix}
u & 0 \\ 
0 &  u^{-1}
\end{pmatrix}\quad d_3 \mapsto\begin{pmatrix}
0 & 1 \\ 
-1 & 0
\end{pmatrix}\quad d_4 \mapsto \begin{pmatrix}
0 & u^{2}\\ 
-u^{-2} & 0
\end{pmatrix} \\
\rho_2~: d_1\mapsto \begin{pmatrix}
0& 1 \\ 
-1 &0
\end{pmatrix}\quad d_2 \mapsto \begin{pmatrix}
v & 0 \\ 
0 &  v^{-1}
\end{pmatrix}\quad d_3 \mapsto\begin{pmatrix}
u & 0 \\ 
0 &u^{-1}
\end{pmatrix}\quad d_4 \mapsto\begin{pmatrix}
v & 0 \\ 
0 &  v^{-1}
\end{pmatrix}\quad \\
\rho_3~: d_1 \mapsto \begin{pmatrix}
0& 1 \\ 
-1 &0
\end{pmatrix}\quad d_2 \mapsto\begin{pmatrix}
0& u^{-1} \\ 
-u &0
\end{pmatrix}\quad d_3 \mapsto\begin{pmatrix}
u & 0 \\ 
0 &u^{-1}
\end{pmatrix}\quad d_4 \mapsto\begin{pmatrix}
u^{-1} & 0 \\ 
0 & u
\end{pmatrix} \\
\rho_4~: d_1 \mapsto\begin{pmatrix}
0& 1 \\ 
-1 &0
\end{pmatrix}\quad d_2 \mapsto\begin{pmatrix}
u & 0 \\ 
0 &u^{-1}
\end{pmatrix}\quad d_3 \mapsto\begin{pmatrix}
u & 0 \\ 
0 &u^{-1}
\end{pmatrix}\quad d_4 \mapsto\begin{pmatrix}
u^{-1} & 0 \\ 
0 & u
\end{pmatrix} \; .
\end{align*}
 
We then turn our attention to the corresponding orbits under the mapping class group action on the $SL_2(\CC)$--character variety of the five punctured sphere, namely the categorical quotient of its variety of representations under the diagonal conjugacy action of $SL_2(\CC)$: 
\begin{displaymath}
\CV(0,5) := \mathrm{Hom}(\Fb_4 , SL_2(\CC) \slash\slash SL_2(\CC) \; .
\end{displaymath}
The links between these orbits and the algebraicity of isomonodromic deformations has been extensively studied by Dubrovin--Mazzocco~\cite{DubMazz}, Boalch~\cite{Boalch}, Cantat--Loray~\cite{CL} and Cousin~\cite{Cousin}. In particular, it is now known that such an orbit is finite if and only if the associated isomonodromic deformation gives rise to an algebraic Garnier solution. We give explicit computations for these, in the form of the following result.

\begin{theoa}\label{thA:MCGOrbits}
Consider the following four families of representations (parametrised by some $u,v,s \in \CC^*$) of the free group over four generators $\Fb_4 := \langle d_1, \ldots, d_4 \, | \, \emptyset \rangle$ into $SL_2(\CC)$.
\begin{align*}
\rho_1& : d_1  \mapsto \begin{pmatrix}
v& 0 \\ 
0 & v^{-1}
\end{pmatrix}\quad  d_2 \mapsto \begin{pmatrix}
u & 0 \\ 
0 &  u^{-1}
\end{pmatrix}\quad d_3 \mapsto\begin{pmatrix}
0 & 1 \\ 
-1 & 0
\end{pmatrix}\quad d_4 \mapsto \begin{pmatrix}
0 & u^{2}\\ 
-u^{-2} & 0
\end{pmatrix} \\
\rho_2& : d_1\mapsto \begin{pmatrix}
0& 1 \\ 
-1 &0
\end{pmatrix}\quad d_2 \mapsto \begin{pmatrix}
v & 0 \\ 
0 &  v^{-1}
\end{pmatrix}\quad d_3 \mapsto\begin{pmatrix}
u & 0 \\ 
0 &u^{-1}
\end{pmatrix}\quad d_4 \mapsto\begin{pmatrix}
v & 0 \\ 
0 &  v^{-1}
\end{pmatrix}\quad \\
\rho_3& : d_1 \mapsto \begin{pmatrix}
0& 1 \\ 
-1 &0
\end{pmatrix}\quad d_2 \mapsto\begin{pmatrix}
0& s^{-1} \\ 
-s &0
\end{pmatrix}\quad d_3 \mapsto\begin{pmatrix}
s & 0 \\ 
0 &s^{-1}
\end{pmatrix}\quad d_4 \mapsto\begin{pmatrix}
s^{-1} & 0 \\ 
0 & s
\end{pmatrix} \\
\rho_4& : d_1 \mapsto\begin{pmatrix}
0& 1 \\ 
-1 &0
\end{pmatrix}\quad d_2 \mapsto\begin{pmatrix}
s & 0 \\ 
0 &s^{-1}
\end{pmatrix}\quad d_3 \mapsto\begin{pmatrix}
s & 0 \\ 
0 &s^{-1}
\end{pmatrix}\quad d_4 \mapsto\begin{pmatrix}
s^{-1} & 0 \\ 
0 & s
\end{pmatrix}
\end{align*}
Then:
\begin{enumerate}
\item the associated points in $\CV(0,5)$ give rise to four pairwise distinct families of length four finite orbits under the pure mapping class group $\PMCG(0,5)$;
\item the families of (non--pure) mapping class group orbits associated with $\rho_1$ and $\rho_2$ are also distinct; however those associated with $\rho_3$ and $\rho_4$ are special cases of $\rho_2$--type orbits. More precisely, this means that for any $s \in \CC^*$ and $i=3,4$ there exist two parameters $(u,v)$ (depending on $s$ and $i$) such that the orbit of the class of $\rho_i$ with parameter $s$ is equal to that of $\rho_2$ with parameters $(u,v)$.
\end{enumerate} 
\end{theoa}

\subsection{Explicit construction of the solutions}

Going from Theorem \ref{thA:MCGOrbits}, we know that our procedure yields two distinct two--parameter families of finite mapping class group orbits. Therefore, the next step should be to try and explicitly construct the associated isomonodromic deformations and algebraic Garnier solutions. Note that the existence, and uniqueness up to gauge transformation, of such families of connections follows from the classical Riemann--Hilbert correspondence. The original part of our work does indeed reside in the fact that we give explicit constructions for these objects and the associated algebraic Garnier solutions.

\paragraph{Solutions associated with $\rho_1$.} 

This deformation has already been quite extensively studied in one of our earlier works~\cite{A2}. We quickly recall the main result of this paper, namely the explicit formulas giving a logarithmic flat connection over $\PP^2_\CC$ whose monodromy representation is exactly $\rho_1$. 

\begin{theo}[Girand \cite{A2}]\label{th:A2}
There exists an \emph{explicit} two--parameter family $\nabla_{\lambda_0, \lambda_1}$ of logarithmic flat connections over the trivial rank two vector bundle $\CC^2 \times \PP^2 \rightarrow \PP^2$ with the following properties:
\begin{enumerate}
\item[(i)] the polar locus of $\nabla_{\lambda_0, \lambda_1}$ is equal to the quintic $\Qr\in \PP^2$ and as such does not depend on $\lambda_0, \lambda_1 \in \CC$;
\item[(ii)] the monodromy of $\nabla_{\lambda_0, \lambda_1}$ is conjugated to $\rho_{u,v}$ with $u = -e^{-i \pi \lambda_0}$ and $v = e^{-i \pi \lambda_1}$. It is a virtually abelian dihedral representation of the fundamental group $\Gamma:=\pi_1(\PP^2-\Qr)$ into $SL_2(\CC)$ whose image is not Zariski--dense.
\end{enumerate}
The connection $\nabla_{\lambda_0, \lambda_1}$ is given in some affine chart $\CC^2_{x,y} \subset	\PP^2$ by:
\begin{displaymath}
\nabla_{\lambda_0,\lambda_1} = \dd - \dfrac{1}{2(x^2+y^2+1-2(xy+x+y))} (\lambda_0 A_0 + \lambda_1	A_1 + A_2) \; , 
\end{displaymath}
where\tiny
\begin{displaymath}
A_0 := \begin{pmatrix}
2(x-1)y\dd x + (x^2 + x(y-2) - y +1)x \frac{\dd y}{y} & 2(2x - y +2)y \dd x + (2x^2 +y(x-y+3)-2)x \frac{\dd y}{y} \\ 
-2y^2 \dd x + (x+y-1)x^2\frac{\dd y}{y}  & -2(x-1)y\dd x - (x^2 + x(y-2) - y +1)x\frac{\dd y}{y}
\end{pmatrix} 
\end{displaymath}
\begin{displaymath}
A_1 := \begin{pmatrix}
(x^2+(x-1)(y-1))y \frac{\dd x}{x} + 2(x-1)x \dd y & (x^2 + y(x-y+3) -2)y \frac{\dd x}{x} + 2(2x-y+2)x \dd y \\ 
-(x+y-1)y^2 \frac{\dd x}{x}-2x^2 \dd y & -(x^2+(x-1)(y-1))y \frac{\dd x}{x} - 2(x-1)x \dd y
\end{pmatrix} 
\end{displaymath}
\begin{displaymath}
A_2 := \begin{pmatrix}
-(x+y+1)y \dd x - (x^2-x(y+2)-y+1)x \frac{\dd y}{y} & -2(x-y+3)y \dd x  -(x^2-2y(x+1)+1) x \frac{\dd y}{y}\\ 
0 & (x+y+1)y \dd x + (x^2-x(y+2)-y+1)x \frac{\dd y}{y}
\end{pmatrix} \; .
\end{displaymath}
\normalsize Moreover, the monodromy representation of the above factors through a curve if and only if there exists $(p, q) \in \ZZ^2\setminus \lbrace (0,0) \rbrace$ such that $p \lambda_0 + q \lambda_1 = 0$. 
\end{theo}

\paragraph{Solutions associated with $\rho_2$.}

In Section \ref{sec:explicitConstruction}, we adapt the tools introduced in \cite{A2} to give an explicit construction of the family of algebraic isomonodromic deformations appearing in Theorem \ref{thA:quintics}. Namely, we give explicit formulas for an algebraic family of logarithmic flat connections on the trivial rank $2$ vector bundle over $\PP^2(\CC)$ whose polar locus is exactly the quintic curve $\Qr^\prime$ given by~: 
\begin{displaymath}
y(y-t)t(x^2-yt) = 0 \; .
\end{displaymath}
The fundamental group of the complement of the curve $\Qr^\prime$ is isomorphic to the second group appearing in Theorem~\ref{thA:quintics}, namely~: 
\begin{displaymath}
\langle a,b,c \, | \, [a,b] = [a,c^{-1}bc] = 1, \, (bc)^2 = (cb)^2 \rangle \; .
\end{displaymath}
In a similar fashion to Theorem \ref{th:A2}, we prove the following result.

\begin{theoa}\label{thA:GarnierSol2}
There exists an \emph{explicit} two--parameter family $\nabla_{\lambda_0, \lambda_1}$ of logarithmic flat connections over the trivial rank two vector bundle $\CC^2 \times \PP^2 \rightarrow \PP^2$ with the following properties:
\begin{enumerate}
\item[(i)] the polar locus of $\nabla_{\lambda_0, \lambda_1}$ is equal to the quintic $\Qr^ \prime\in \PP^2$ defined by the equation (in homogeneous coordinates $[x:y:t]$)
\begin{displaymath}
y(y-t)t(x^2-yt) = 0 \;~;
\end{displaymath}
\item[(ii)] the monodromy of $\nabla_{\lambda_0, \lambda_1}$ is conjugated to $\rho_{u,v}$ with $u = e^{i \pi \lambda_0}$ and $v = e^{i \pi \lambda_1}$. It is a virtually abelian dihedral representation of the fundamental group $\Gamma_2:=\pi_1(\PP^2-\Qr)$ into $SL_2(\CC)$ whose image is not Zariski--dense.
\end{enumerate}
The connection $\nabla_{\lambda_0, \lambda_1}$ is given in the affine chart $\CC^2_{x,y} \subset	\PP^2$ by:
\begin{displaymath}
\nabla_{\lambda_0,\lambda_1} = \dd - \dfrac{1}{y(y-1)(x^2-y)}\Omega_{\lambda_0,\lambda_1} \, , 
\end{displaymath}
where {\scriptsize
\begin{displaymath}
\Omega_{\lambda_0,\lambda_1} := \begin{pmatrix}
-\frac{(y-1)(x^2-y)}{4y}\dd y & - \dfrac{2\lambda_0y(y-1)\dd x + ( \lambda_0x(1-y) + \lambda_1(x^2-y) )\dd y}{2}y \\ 
- \dfrac{2\lambda_0y(y-1)\dd x + ( \lambda_0x(1-y) + \lambda_1(x^2-y) )\dd y}{2} & \frac{(y-1)(x^2-y)}{4y}\dd y 
\end{pmatrix} \; . 
\end{displaymath}}
Moreover, the monodromy representation of such a connection factors through an orbicurve if and only if there exists $(p, q) \in \ZZ^2\setminus \lbrace (0,0) \rbrace$ such that $p \lambda_0 + q \lambda_1 = 0$.
\end{theoa}

\section{Classification of non--degenerate representations}\label{sec:Quintics}

\subsection{Representations factoring through a curve}

Representations of fundamental groups of quasi-projective varieties in $SL_2(\CC)$ have been classified mainly by Corlette and Simpson~\cite{CorSim}. One important class of such representations is that of those factoring through a curve.

\begin{defi}\label{def:factorCurve} ~\cite{CorSim, LTP}
Let $\Gamma$ be the fundamental group of the complement of some curve in $\PP^2(\CC)$. We say that a representation $\rho : \Gamma \rightarrow PSL_2(\CC)$ \emph{factors through a curve} if there exists a complex projective curve $C$, a divisor $\Delta$ (resp. $\delta$) in $\PP^2$ (resp. $C$), an algebraic mapping $f : \PP^2 -\Delta \rightarrow C - \delta$ and a representation $\tilde{\rho}$ of the fundamental group of $C-\delta$ into $PSL_2(\CC)$ such that
\begin{enumerate}
\item[(i)]  $\Delta$ contains $\Qr$, therefore there exists a natural group homomorphism $m : \pi_1(\PP^2 - \Delta) \rightarrow \Gamma_2$;
\item[(ii)] the diagram
\begin{displaymath}
\xymatrix{
\pi_1(C-\delta)  \ar[rd]_{\tilde{\rho}} & \ar[l]_{f_*} \pi_1(\PP^2 - \Delta) \ar[d]^{\rho\circ m}\\
 & PSL_2(\CC) 
}
\end{displaymath}
commutes.
\end{enumerate}
Moreover, if some representation $\varrho : \Gamma \rightarrow SL_2(\CC)$ is such that $P\circ \rho : \Gamma \rightarrow SL_2(\CC)$, we will say that $\varrho$ \emph{factors through an orbicurve}.
\end{defi}

Indeed, representations admitting such a factorisation can be obtained through pullback from the monodromy of some logarithmic flat connection on a curve~\cite{LTP}. Moreover, we have the following refinement by Loray, Pereira and Touzet~\cite{LTP} of a theorem by Corlette and Simpson~\cite{CorSim}.

\begin{theo}[Corlette--Simpson, Loray--Pereira--Touzet]\label{th:LPT-CS}
Let $X$ be a quasi--projective surface. Then any non--rigid representation $\rho : \pi_1(X) \rightarrow PSL_2(\CC)$ with Zariski--dense image factors through a curve.
\end{theo}

\begin{rema}
This implies that any representation satisfying condition \textbf{(C1}) that is neither rigid nor dihedral factors through a curve.
\end{rema}

\subsection{Understanding the list}

In~\cite{Degtyarev}, Degtyarev classifies the quintic curves in the projective plane $\PP^2(\CC)$ whose complement has non--abelian fundamental group. In this work we are interested in infinite groups giving rise to representations satisfying conditions \textbf{(C1}) and \textbf{(C2}), so we shall recall the parts of the aforementioned list of interest to us, starting by briefly recalling the notations used.

\subsubsection{Groups appearing in the list}

\paragraph{Toric groups.}Toric group are the family of fundamental groups of toric links, defined as follows:
\begin{itemize}
\item for $r \geq 1$, set
\begin{displaymath}
 T_{2,2r} := \langle a,b \, | \, (ab)^r = (ba)^r \rangle \; ; 
 \end{displaymath} 
 \item if $p$ and $q$ are two relatively prime integers set
 \begin{displaymath}
 T_{p,q} := \langle a,b \, | \, a^p = b^q \rangle \; .
 \end{displaymath}
\end{itemize}

\paragraph{Braid groups.}Recall that the braid group on $p$ strands is given by:
\begin{displaymath}
B_p := \langle \sigma_1, \ldots , \sigma_{p-1} \, | \, [\sigma_i, \sigma_j] = 1 \text{ if } |i-j| > 1, \sigma_i \sigma_{i+1}\sigma_i = \sigma_{i+1} \sigma_i \sigma_{i+1} \rangle \; .
\end{displaymath}

\paragraph{"$G$--type groups".}Let $p$ be a prime number and $T \in \ZZ[t]$ be some integral polynomial; then we define the groups $G(T)$ and $G_p(T)$ as the extensions:
\begin{displaymath}
0 \rightarrow \ZZ[t] \slash (T) \rightarrow G(T) \rightarrow \ZZ \rightarrow 0
\end{displaymath}
and 
\begin{displaymath}
0 \rightarrow \FF_p[t] \slash (T) \rightarrow G_p(T) \rightarrow \ZZ \rightarrow 0
\end{displaymath}
where the conjugation action of the generator of the quotient on the kernel is the multiplication by $t$. Note that these groups are solvable, therefore we will be able to use the following well--known result.

\begin{prop} \label{prop:reprSolv}
Let $G$ be a solvable group. Then any irreducible group representation $\rho : G\rightarrow PSL_2(\CC)$ with infinite image is dihedral.
\end{prop}

\paragraph{Artin groups.}We will also need a few Artin groups, namely:
\begin{itemize}
\item $\Ar^1(p,q,r) := \langle a,b,c \, | \, a^p = b^q = c^r =abc \rangle$; 
\item $\Ar^2(p,q,r) := \langle a,b,c \, | \, a^p = b^q = c^r=abc=1 \rangle$;
\item $\Ar^3(p,q,r) := \langle a,b \, | \, a^p = b^q = 1, (ab) ^r = (ba)^r \rangle$.
\end{itemize}

\paragraph{"Unusual" groups. }Here we shall list some exceptional groups arising in Degtyarev's classification by giving finite presentations obtained through the Zariski--Van Kampen method.
\begin{itemize}
\item $\Gamma_5 := \langle u,v \, | u^3 = v^7 =(uv^2)^2 \rangle$;
\item $\Gamma_4 := \langle a,b,c \, | \, aba=bab,\, cbc=bcb, \, abcb^{-1}a = bcb^{-1}abcb^{-1} \rangle \; $;
\item $\Gamma_3 := \langle a,b \, | \, [a^3,b] = 1 \, , ab^2=ba^2 \rangle$;
\item $\Gamma_3^\prime := \langle a,b,c \, | \, aca=cac,\, [b,c]=1, \, (ab)^2=(ba)^2 \rangle$;
\item $\Gamma_2 := \langle a,b,c \, | \,  [a,b]=[a,c^{-1}bc] = 1, \, (bc)^2=(cb)^2 \rangle$;
\item $\Gamma_2^\prime := \langle a,b,c \, | \, (ab)^2=(ba)^2,\,(ac)^2=(ca)^2,\, [b,c]=1 \rangle$.
\end{itemize}

\subsubsection{Arnol'd's notation for curve singularities}

In order to allow for an efficient listing of the curves involved in Degtyarev's classification we will make use of Arnol'd's notation for curves singularities~\cite{AGV}. More specifically: 
\begin{itemize}
\item we will denote by $A_p$ a singular point of local type
\begin{displaymath}
x^2 + y^{p+1} = 0 \quad ;
\end{displaymath}
\item we will denote by $E_6$ a singular point of local type
\begin{displaymath}
x^3 + y^{4} = 0 \quad ;
\end{displaymath}
\item we will say that a curve is \emph{of type $C_{k_1}\sqcup \ldots \sqcup C_{k_n}$} if it (globally) has $n$ irreducible components of respective degrees $k_1, \ldots, k_n$. In the case where several such components have the same degree, we will use the shorthand notation $m C_{k_\ell}$;
\item if any degree $d$ irreducible component of our curve has singular points, we shall denote it by writing $C_d (\Sigma)$, where $\Sigma$ is the list of the aforementioned curve's singularities, of the form $m_1 A_{p_1} \sqcup \ldots \sqcup m_n A_{p_n}$ (note that if the curve is smooth, we will simply use the shorthand $C_d$ instead of $C_d(\emptyset)$);
\item finally, we will use the following notation regarding the mutual position of two irreducible curves $C$ and $C^\prime$:
\begin{itemize}
\item $\times d$ if $C^\prime$ intersects $C$ with multiplicity $d$ at a non--singular point of $C$;
\item $A_p$ if $C^\prime$ intersects $C$ transversally at a singular point of $C$ of type $A_p$;
\item $A_p^*$ if $C^\prime$ is tangent (with smallest possible multiplicity) to $C$ transversally at a singular point of $C$ of type $A_p$.
\end{itemize}
Moreover, if the curve is of type $C_3\sqcup 2C_1$ and the two lines intersect at one of the special points described above, we shall underline it in the list.
\end{itemize}

\begin{example}
A curve of type $C_3 (A_1) \sqcup C_1$ with intersection $\times 2,\, \times 1$ would be made up of a nodal cubic and a line, the latter intersecting the former at two smooth points, once with a tangent and the other transversally. The same curve with intersection $A_1,\, \times 1$ would have the line intersect the cubic transversally at both some smooth point and the nodal singular point.
\end{example}

\subsubsection{The list} 

\begin{table}[!!h]
\begin{center}
\begin{tabular}{ccc}
\textbf{Curve type} & \textbf{Intersection type(s)} &\textbf{Group(s)} \\ 
\hline 
$C_5(A_6 \sqcup 3A_2)$ & --- & $\Gamma_5$ \\ 
$C_4(3A_2)\sqcup C_1$ & $\times2, \times 2$ & $\Gamma_4$ \\ 
 & $\times 2, \times 1, \times 1$ or $A_2^*, \times 1$ & $B_3$ \\ 
 & else & $G_3(t+1)$ \\
$C_4(2A_2 \sqcup A_1)\sqcup C_1$ & $\times 4$ & $B_4$ \\ 
  & $\times 2, \times 2$ & $B_3$ \\ 
$C_4(2A_2)\sqcup C_1$& $\times 4$ or $\times 2, \times 2$ & $B_3$ \\ 
$C_4(A_4 \sqcup A_2)\sqcup C_1$ & $\times 3, \times 1$ & $\ZZ \times \Ar^2(2,3,5)$ \\ 
  & $A_4^*$ & $B_3$ \\ 
  & $A_2, \times 2$ & $G_5(t+1)$ \\ 
$C_4(A_2 \sqcup A_2)\sqcup C_1$ & $A_2,\times 2$ & $B_3$ \\ 
$C_4(A_6)\sqcup C_1$ & $A_6,\times 2$ & $B_3$ \\ 
$C_4(A_5)\sqcup C_1$ & $\times 4$ or $\times 2, \times 2$ & $B_3$ \\ 
$C_4(E_6)\sqcup C_1$ & $\times 4$ & $T_{3,4}$ \\ 
  & $\times 2, \times 2$ & $B_3$ \\ 
$C_3(A_2) \sqcup C_2$ & $\times 3, \times 3$ & $\Gamma_3$ \\ 
$C_3(A_2) \sqcup 2C_1$ & $\times 3 \; ; \times 2, \times 1$ & $\Gamma_3^\prime$ \\ 
  & $\times 3 \; ; A_2^*$ & $T_{2,6}$ \\ 
  & $\underline{\times 3} \; ; \underline{\times 1}, A_2$ & $T_{2,4}$ \\ 
  &  $\underline{\times 3} \; ; \underline{\times 1},\times 1, \times 1$& $\ZZ \times B_3$ \\ 
  &  $\times 3 \; ; A_2 \times 1$& $\ZZ \times B_3$ \\ 
  &  $\times 3 \; ; \times 1,\times 1, \times 1$& $\ZZ \times B_3$ \\ 
  &  $\underline{\times 2}, \times 1 \; ; \underline{\times 1},\times 2$& $\ZZ \times B_3$ \\ 
  & $A_2,\underline{\times 1} \; ; \underline{\times 1},\times 2$ & $G(t^2-1)$\\
$C_3(A_1)\sqcup 2C_1$ & $\times 3 \; ; \times 3$ & $G(t^3-1)$ \\ 
  & $\underline{\times 3} \; ; \underline{\times 1}, \times 2$ & $G(t^2-1)$ \\ 
  & $\underline{\times 1}, \times 2 \; ; \underline{\times 1}, \times 2$ & $G(t^2-1)$ \\ 
$2C_2 \sqcup C_1$ & the two $C_2$ intersect with multiplicity $4$	 & $\Fb_2$, $T_{2,4}$ \\ 
  & the two $C_2$ intersect at two points & $\Fb_2$, $T_{2,4}$ \\ 
 & the two $C_2$ intersect with multiplicity $3$ & $\ZZ \times B_3$ \\ 
$C_2 \sqcup 3C_1$ & the three $C_1$ have a common point & $\Gamma_2, \, \ZZ \times \Fb_2$ \\ 
  & else & $\Gamma_2^\prime, \, \ZZ \times \Fb_2, \, \ZZ \times T_{2,4}$ \\ 
$5C_1$ & the five $C_1$ have a common point & $\Fb_4$ \\ 
  & there is a quadruple point & $\ZZ \times \Fb_3$ \\ 
  & there are two triple points & $\Fb_2 \times \Fb_2$ \\ 
  & there is only one triple point & $\ZZ\times\ZZ\times\Fb_2$ \\ 
\end{tabular} 
\end{center}
\caption{Degtyarev's list}\label{tab:Deg}
\end{table}

Table~\ref{tab:Deg} sums up Degtyarev's list; for more details we refer the reader to the original paper~\cite{Degtyarev}. Note that we have eliminated all finite groups from the list before reproducing it.

We shall go through Degtyarev's list in three steps: 
\begin{itemize}
\item first, we investigate curves with large singularities, allowing us to remove some of them from the list;
\item secondly, we eliminate any group in the list which cannot yield a representation satisfying conditions \textbf{(C1}) and \textbf{(C2}) for purely algebraic reasons;
\item finally, we review the remaining fundamental groups in the list by curve type.
\end{itemize} 
Doing so, we will make heavy use of the following elementary fact about $PSL_2(\CC)$ (for a proof, see Lemma 1.2.3 in \cite{AThese}).

\begin{lemm} \label{lem:centralPSL}
Let $M \in PSL_2(\CC)$ centralising some non--abelian subgroup $G \leq PSL_2(\CC)$; then $M$ is the identity element of $PSL_2(\CC)$.
\end{lemm}

\subsection{Large singularities}

The goal of this paragraph is to study the special cases where the quintic curve has some large order (\emph{i.e} greater or equal to three) singular point. In fact, we prove that representations of the fundamental group of the complement of such curves almost always factor through some curve.

\subsubsection{Five--fold singularities}

First, we can eliminate the case of the curve $5C_1$ with a quintuple point through a rather elementary reasoning, as evidenced by the result below.

\begin{prop}\label{prop:fiveFoldSing}
Let $\Qr$ be a quintic in the projective plane $\PP^2(\CC)$ of type $5C_1$ such that its five irreducible components share a common point. Then any representation $\rho : \pi_1 (\PP^2 - \Qr) \rightarrow PSL_2(\CC)$ with non--abelian image factors through a curve.
\end{prop}
\begin{proof}
Since all irreducible components of $\Qr$ intersect at a common point then by blowing it up we get a locally trivial fibration 
\begin{displaymath}
\xymatrix{
\hat{\PP^2} - \hat{\Qr}  \ar[d]^{f} & \ar[l] \CC \\
B & 
}
\end{displaymath}
where $B$ is isomorphic to $\PP^1$ minus five points and $\hat{\Qr}$ is the total transform of the quintic $\Qr$. The homotopy exact sequence associated with $f$ then yields:
\begin{displaymath}
0 = \pi_1(\CC) \rightarrow \pi_1(\hat{\PP^2} - \hat{\Qr}) \rightarrow \Fb_4 = \pi_1(B) \rightarrow \pi_0(\CC) \;.
\end{displaymath}
Since $\pi_1 (\PP^2 - \Qr) \cong \pi_1(\hat{\PP^2} - \hat{\Qr})$ (as we blew up a point in $\Qr$) then one gets that $\pi_1 (\PP^2 - \Qr) \cong \pi_1(B)$ thus $\rho$ factors through $B$ by means of the morphism $f$.
\end{proof}

\subsubsection{Quadruple singularities}

In the rare case where a quintic curve in Degtyarev's list has a quadruple singularity, we can discard it using the following result.

\begin{prop}
Let $\Qr$ be a quintic in the projective plane $\PP^2(\CC)$ of type $5C_1$ such that exactly four of its irreducible components share a common point. Then any representation $$\rho : \pi_1 (\PP^2 - \Qr) \rightarrow PSL_2(\CC)$$ with non--abelian image factors through a curve.
\end{prop}
\begin{proof}
Up to an element in $PGL_3(\CC)$, one can assume that the line that does not intersect the other four at their common point is the line at infinity and that the aforementioned quadruple point is the origin of the corresponding affine chart (identified with $\CC^2$). One gets a locally trivial fibration $\psi : \PP^2 - \Qr \rightarrow \PP^1_4$ defined as follows: for any point $x \in \PP^2 - Q$, $\psi(x)$ the line going through the origin and $x$. This fibration has fibre equal to $\CC^*$ and a natural section given by any line not contained in $\Qr$ nor going through the origin, the latter giving us an homeomorphism 
\begin{align*}
f : \PP^2-Q& \xrightarrow[]{\sim} \PP^1_4 \times \CC^*\\
x&\mapsto \left(\psi(x), \tau(x) \right) \; ,
\end{align*} 
where $\tau(x)$ is the slope of the line $\psi(x) \subset \CC^2$. This means that $\pi_1(\PP^2-Q) \cong \Fb_3 \times \ZZ$. The fact that the image of $\rho$ must be non--abelian dictates that the image of one of these factors must also be. As such, any element in the image of the other one centralises a non--abelian subgroup in $PSL_2(\CC)$ and os must be trivial by Lemma~\ref{lem:centralPSL}. Thus one gets that $\rho$ factors through $\PP_4^1$ (as $\pi_1(\CC^*) = \ZZ$ is abelian).

\end{proof}

\subsubsection{Triple singularities}

It remains to see what happens when the studied curve has a singular point of order exactly three.

\begin{prop}\label{prop:tripSing}
Let $\Qr$ be a quintic in the projective plane $\PP^2(\CC)$. Assume that there exists a triple point $p_0 \in \Qr$, \emph{i.e} that for any line $L$ in $\PP^2$ containing $p_0$ one has $\card(L \cap (\Qr-p_0))= 2$ (counted with multiplicity); then any representation $\rho : \pi_1 (\PP^2 - \Qr) \rightarrow PSL_2(\CC)$ with nonabelian, non--dihedral infinite image factors through a curve up to pull--back by a double covering.
\end{prop}
\begin{proof}
Start by blowing up $p_0$; in the following we will denote by $\hat{C}$ (resp. $\tilde{C}$) the total (resp. strict) transform of any curve $C \subset \PP^2$ by this blow--up. This turns the pencil of lines going through the point $p_0$ into an actual fibration
\begin{displaymath}
\xymatrix{
\hat{\PP^2}  \ar[d]^{f} & \ar[l] \PP^1 \\
B & 
}
\end{displaymath}
endowed with a natural section $\tau$ given by the exceptional divisor $E(p_0)$ (note that both $E(p_0)$ and $B$ are isomorphic to the projective line $\PP^1$). Since a generic line containing $p_0$ in $\PP^2$ cuts $\Qr-p_0$ at two distinct points, we also get a "double section" of this fibration, namely a ramified double--covering $r : \Cr \xrightarrow[]{2:1} \PP^1$ and a mapping $\sigma : \Cr \rightarrow \hat{\PP^2}$ such that the diagram 
\begin{displaymath}
\xymatrix{
& \hat{\PP^2}  \ar[d]^{f}  \\
\Cr \ar[r]^{r} \ar[ru]^\sigma & B 
}
\end{displaymath}
commutes. Note that by construction $\sigma(\Cr)$ is a component of the strict transform pf $\Qr$ ; in the case where it is reducible we have a much stronger result (see Remark~\ref{rem:noCoverNeeded}).

Consider the fibre product $S := \hat{\PP^2} \times_{B} \Cr$; by definition, this is exactly the set $$\lbrace (p,q) \in \hat{\PP^2}\times \Cr \, | \, f(p) = r(q) \rbrace\, ,$$ giving us another commutative diagram (and a natural mapping $\chi : S \rightarrow \hat{\PP}^2$):
\begin{displaymath}
\xymatrix{
S  \ar[d]^{\pi} \ar[r]^\chi &  \hat{\PP^2} \ar[d]^f \\
C \ar[r]^r_{2:1} & B
} 
\end{displaymath}
The second projection $\pi : S \rightarrow C$ yields a fibration 
\begin{displaymath}
\xymatrix{
S  \ar[d]^{\pi} & \ar[l] \PP^1 \\
C & 
}
\end{displaymath}
endowed with three distinct sections $\sigma_0, \sigma_1$ and $\sigma_\infty$. Indeed one can define two sections of $\pi$ using the double section $\sigma$ as follows: for any $x \in C$ there exists $y \in C$ (generically distinct from $x$) such that $r(x) = r(y)$, so we set $\sigma_0(x) :=  (\sigma(x),x)$ and $\sigma_0(x) :=  (\sigma(y),x)$ (this is well defined since $f \circ \sigma = r$). The third section $\sigma_\infty$ is then given by the exceptional divisor: $x \mapsto \tau \circ r (x)$. Note that there are finitely many points $x \in \Cr$ such that $\sigma_i(x) = \sigma_j(x)$ for some $i \neq j$ ($i,j = 0,1,\infty$); indeed these special points come from either singular points of $\Qr - p_0$, fibres of $f$ tangent to $\Qr$ or possible intersections between $E(p_0)$ and the strict transform of $\Qr$ in $\hat{\PP^2}$. 

One can now establish a birational mapping $h$ between $S$ and the direct product $\PP^1 \times \Cr$. Let $x \in C$ be some generic point (so that $\sigma_0(x), \sigma_1(x)$ and $\sigma_\infty(x)$ are pairwise distinct) and define $h$ on $L$ as the homography sending $\sigma_0(x)$ (resp. $\sigma_1(x), \sigma_\infty(x)$) onto $0$ (resp. $1, \infty$). This defines a rational mapping inducing an isomorphism between a Zariski--open set in $S$ and one in $\Cr \times \PP^1$.

Let us now look at the above recipe in topological terms: since we are blowing up a point in $\Qr$ then $\pi_1(\PP^2 - \Qr) \cong \pi_1(\hat{\PP^2}-\hat{\Qr})$. Then remark that we can extend our representation $\rho$ to the fundamental group of the complement of $\hat{\Qr} \cup L_1 \cup \ldots \cup L_k$, where the $L_i$ are the fibres of $f$ that are either fibres above the ramification locus of $r$ or contain points of the indeterminacy locus of $h$ or $h^{-1}$, by setting the image of any simple loop around any special fibre $L_i$ that is not a component of $\hat{\Qr}$ to be the identity matrix $I_2$. 

This means that we have a new representation $\tilde{\rho} : \pi_1( \hat{\PP^2}-(\hat{\Qr} \cup L_1 \cup \ldots \cup L_k) \rightarrow PSL_2(\CC)$ such that $\mathrm{Im}(\tilde{\rho}) = \mathrm{Im}(\rho)$. The covering $r$ being non--ramified above $B-f( L_1 \cup \ldots \cup L_k)$ one gets that the fundamental group $G := \pi_1(S-\chi^*(\hat{\Qr} \cup L_1 \cup \ldots \cup L_k))$ is isomorphic to an index at most two subgroup of $\pi_1( \hat{\PP^2}-(\hat{\Qr} \cup L_1 \cup \ldots \cup L_k))$. Thus, by restriction, one gets a representation $\rho_G : G \rightarrow PSL_2(\CC)$.

Since the birational map $h$ only blows up and/or contract divisors contained in $$\chi^*(\hat{\Qr} \cup L_1 \cup \ldots \cup L_k), $$ there is an isomorphism between $G$ and the fundamental group $G^\prime$ of $\Cr \times \PP^1 - \delta$, where $\delta$ is the strict transform under $h$ of $\chi^*(\hat{\Qr} \cup L_1 \cup \ldots \cup L_k)$ (\emph{i.e} the Zariski--closure of its image under $h$). This divisor $\delta$ is by construction a finite union of "vertical" and/or "horizontal" curves in $\Cr \times \PP^1$ thus $\Cr \times \PP^1 - \delta = (\Cr - \eta) \times (\PP^1 - \eta^\prime)$, where $\eta$ (resp. $\eta^\prime$) is a finite divisorial sum of points in $\Cr$ (resp. $\PP^1$). This in turn implies that $G^\prime \cong \pi_1(\Cr- \eta) \times \pi_1 (\PP^1 - \eta^\prime)$.

Since $\mathrm{Im}(\rho_G)$ is an index at most $2$ subgroup in the non--dihedral group $\mathrm{Im}(\rho)$, it must be non--abelian. This forces the image of one of the factors in $G$ to be non--abelian therefore any element in the image of the other one centralises a non--abelian subgroup in $PSL_2(\CC)$ and so is trivial by Lemma~\ref{lem:centralPSL}. As such, $\rho_G$ factors through either $\Cr- \eta$ or $\PP^1-\eta^\prime$, using either the first or the second projection. 
\end{proof}

\begin{rema}\label{rem:noCoverNeeded}
If the "double section" $\sigma$ in the proof above is actually made up of two well--defined sections (i.e if $\sigma(\Cr)$ is a reducible component of $\hat{\Qr}$), then one naturally does not need the non--dihedral hypothesis. Indeed, $\sigma(\Cr)$ must be made out of two lines and so we get three well--defined sections of the initial fibration, without needing to go through a double covering. Thus, any representation $\rho : \pi_1 (\PP^2 - \Qr) \rightarrow PSL_2(\CC)$ with nonabelian, infinite image factors through a curve.
\end{rema}

The following curves have a triple singularity that falls under Remark~\ref{rem:noCoverNeeded}, therefore we can remove them from the list:

\begin{itemize}
\item $C_4(E_6) \sqcup C_1$;
\item $C_3(A_2)\sqcup 2 C_1$ with intersection types $\times 3 \; ; A_2^*$, $\underline{\times 3} \; ; \underline{\times 1},A_2$ and $\underline{\times 1}, \times 2 \; ; \underline{\times 1}, A_2$;
\item $2C_2 \sqcup C_1$ if either the two conics intersect with multiplicity $4$ and the line is their common tangent at this point ($\Fb_2$--case in the list) or if they have two intersection points through which the line passes ($\Fb_2$--case);
\item $C_2 \sqcup 3 C_1$ if two of the line are tangent to the conic and the third passes through both tangency points;
\item $5C_1$ with any number of triple points.
\end{itemize}

In the first two cases, one needs to chose the singular point of the irreducible component with highest degree as the singularity in Proposition~\ref{prop:tripSing}. For example, when one looks at a curve of type $C_4(E_6) \sqcup C_1$ then one sees that any line going through the $E_6$ type singular point intersects the quintic in two distinguishable other points: one on the quartic $C_4$ and one on the line $C_1$. Thus, after blowing up the aforementioned singular point, the pencil of lines going through it becomes a fibration endowed with two sections. Thus Lemma~\ref{lem:centralPSL} forces the image of one of the factors in $G$ to be non--abelian therefore any element in the image of the other one centralises a non--abelian subgroup in $PSL_2(\CC)$ and so is trivial. As such, $\rho_G$ factors through either $\Cr- \eta$ or $\PP^1-\eta^\prime$, using either the first or the second projection.

To treat the case of $2C_2 \sqcup C_1$, take any intersection between the line and quadric and for $C_2 \sqcup 3 C_1$ and $5C_1$ take any intersection between exactly three components. Note that this means that we can completely eliminate the free group on two generators $\Fb_2$ from the list.

\subsection{Eliminating groups}

In this paragraph, we eliminate several groups that, for strictly algebraic reasons, cannot give rise to a representation satisfying conditions \textbf{(C1}) and \textbf{(C2}). More precisely, we prove the following result.

\begin{prop}\label{prop:DegGroups}
Let $G$ be one of the following groups:
\begin{itemize}
\item the braid groups $B_3$ and $B_4$;
\item the group $G(t^3-1)$;
\item the groups $G_p(t+1)$ for some prime number $p$;
\item the Artin group $\Ar^2(2,3,5)$;
\item the "unusual" groups $\Gamma_4, \, \Gamma_3$ and $\Gamma_3^\prime$;
\item any direct product of $\ZZ$ and one of the above.
\end{itemize}
Let $\rho$ be a representation of $G$ into $PSL_2(\CC)$. Then $\rho$ cannot satisfy both conditions \textbf{(C1}) and \textbf{(C2}).
\end{prop}

\subsubsection{Braid groups} \label{subsubsec:braidGroups}

\paragraph{On three strands.} Consider the group $B_3= \langle \sigma_1, \sigma_{2} \, | \, \sigma_1 \sigma_2\sigma_1 = \sigma_2 \sigma_1 \sigma_2 \rangle \;$ ; we prove the two following lemmas. 

\begin{lemm} \label{lem:braidGroupDihedral}
There is no representation $\rho : B_3 \rightarrow PSL_2(\CC)$  with non--abelian dihedral image.
\end{lemm}
\begin{proof}
Let $A$ (resp. $B$) be the image of $\sigma_1$ (resp. $\sigma_2$) by $\rho$. Then one has $ABA=BAB$ and thus $A = (AB)^{-1}B(AB)$ is conjugate to $B$. Moreover, one can easily check that $(AB)^3$ commutes with the whole (non abelian) image of $\rho$ and so must be trivial (Lemma~\ref{lem:centralPSL}).
Since $\rho$ is dihedral then since its image must be non abelian one must have (up to global conjugacy and permuting $A$ and $B$) either
\begin{displaymath}
A = \begin{pmatrix}
0 & -1 \\ 
1 & 0
\end{pmatrix} \quad \text{ and } \quad B = \begin{pmatrix}
\lambda & 0 \\ 
0 & \lambda^{-1}
\end{pmatrix}
\end{displaymath}
or
\begin{displaymath}
A = \begin{pmatrix}
0 & -1 \\ 
1 & 0
\end{pmatrix} \quad \text{ and } \quad B = \begin{pmatrix}
0 & -\lambda \\ 
 \lambda^{-1} & 0
\end{pmatrix}
\end{displaymath}
for some $\lambda \in \CC^\setminus \lbrace 0,-1,1 \rbrace$ (if $\lambda = \pm 1$ then the image of $\rho$ would be abelian). In the first case, since $A$ and $B$ are conjugate, $\lambda$ must be equal to $\pm i$ and so the groupe generated by $A$ and $B$ is finite abelian. In the second case, the condition one can only have $ABA=BAB$ if $\lambda = \pm 1$ \emph{i.e} if $A=B$, meaning that the image of $\rho$ must be abelian.
\end{proof}

\begin{lemm} \label{lem:braidGroupNonDihedral}
Let $\rho : B_3 \rightarrow PSL_2(\CC)$ be a representation with irreducible non--dihedral image. Then $\rho$ cannot be rigid.
\end{lemm}
\begin{proof}
As above, let $A$ (resp. $B$) be the image of $\sigma_1$ (resp. $\sigma_2$) by $\rho$. Since $\rho$ is not dihedral, then it is Zariski--dense and so must be rigid by Theorem~\ref{th:LPT-CS}. Up to global conjugacy, one can assume that either
 \begin{displaymath}
 A=\begin{pmatrix}
1 & 1 \\ 
0 & 1
\end{pmatrix} \text{ or } A=\begin{pmatrix}
u & 0 \\ 
0 & u^{-1}
\end{pmatrix} \text{ for some } u \in \CC\setminus \lbrace 0,-1,1 \rbrace \; .
\end{displaymath}
In the first case, the facts that $A$ and $B$ are conjugate and $\mathrm{Im}(\rho)$ must be non--abelian forces 
\begin{displaymath}
B = \begin{pmatrix}
1 & 0 \\ 
u & 1
\end{pmatrix}
\end{displaymath} 
for some $u \in \CC^*$, but then for any such $u$ one has $ABA = BAB$ in $PSL_2(\CC)$ if and only if $u = -1$. However, one easily checks that for any $v \in \CC^*$ the following defines a representation of $B_3$:
\begin{displaymath}
\sigma_1 \mapsto \begin{pmatrix}
v & 1 \\ 
0 & v^{-1}
\end{pmatrix}, \quad \sigma_2 \mapsto \begin{pmatrix}
v^{-1} & 0 \\ 
-1 & v
\end{pmatrix}
\end{displaymath}
and so $\rho$ is not rigid. In the second case, one can check that the equation $B = (AB)^{-1}B(AB)$ (in the variable $B$) admits a solution for any $u\neq 0, \pm 1$ and that such a pair $(A,B)$ defines a representation of $B_3$ for any such $u$. Therefore, the representation cannot be rigid.
\end{proof}

The combination of Lemmas~\ref{lem:braidGroupDihedral} and~\ref{lem:braidGroupNonDihedral} with Theorem~\ref{th:LPT-CS} yield that no representation $\rho : B_3 \rightarrow PSL_2(\CC)$ may satisfy conditions \textbf{(C1)} and \textbf{(C2)}.

\paragraph{On four strands.}

Consider the group $$B_4= \langle \sigma_1, \sigma_{2}, \sigma_3 \, | \, [\sigma_1,\sigma_3] ,\, \sigma_1 \sigma_2\sigma_1 = \sigma_2 \sigma_1 \sigma_2, \,\sigma_3 \sigma_2\sigma_3 = \sigma_2 \sigma_3 \sigma_2 \rangle \;$$ and let $A$ (resp. $B$, $C$) be the image of $\sigma_1$ (resp. $\sigma_2$, $\sigma_3$) by $\rho$; the braid relations then give us:
\begin{displaymath}
[A,C] = I_2, \quad ABA=BAB \quad \text{ and } \quad BCB = CBC \; \;.
\end{displaymath}

If either $A$ or $C$ is trivial then $\rho$ factors through a representation of $B_3$ and so cannot satisfy conditions \textbf{(C1)} and \textbf{(C2)}. Moreover, $B$ cannot be trivial since the image of $\rho$ must be non--abelian, nor can $A$ commute to $B$.

All three matrices must be conjugate as $B = (AB)^{-1}A(AB)$ and $C = (BC)^{-1} B (BC)$. As such, if $\rho$ is dihedral then by applying Lemma~\ref{lem:braidGroupDihedral} we get that both the groups $\langle A,B \rangle$ and $\langle B,C \rangle$ must be finite; combined with the fact that $A$ and $C$ commute, this forces $\mathrm{Im}(\rho)$ to be finite.

Else, we proceed again as above: up to global conjugacy, one can assume that either
 \begin{displaymath}
 A=\begin{pmatrix}
1 & 1 \\ 
0 & 1
\end{pmatrix} \text{ or } A=\begin{pmatrix}
u & 0 \\ 
0 & u^{-1}
\end{pmatrix} \text{ for some } u \in \CC\setminus \lbrace 0,-1,1 \rbrace \; .
\end{displaymath}
In the first case, the facts that $A$, $B$ and $C$ are conjugate and $\mathrm{Im}(\rho)$ must be non--abelian forces (up to global conjugacy)
\begin{displaymath}
B = \begin{pmatrix}
1 & 0 \\ 
u & 1
\end{pmatrix} \text{ and } C=  \begin{pmatrix}
1 & v \\ 
0 & 1
\end{pmatrix}
\end{displaymath} 
for some $u,v \in \CC^*$, but then the relation $ABA = BAB$ (resp. $BCB=CBC$) in $PSL_2(\CC)$ if and only if $u = -1$ (resp. $v=1$). However, one easily checks that for any $w \in \CC^*$ the following defines a representation of $B_4$:
\begin{displaymath}
\sigma_1 \mapsto \begin{pmatrix}
w & 1 \\ 
0 & w^{-1}
\end{pmatrix}, \quad \sigma_2 \mapsto \begin{pmatrix}
w^{-1} & 0 \\ 
-1 & w
\end{pmatrix}, \quad \sigma_3 \mapsto \begin{pmatrix}
w & 1 \\ 
0 & w^{-1}
\end{pmatrix},
\end{displaymath}
and so $\rho$ is not rigid. In the second case, since $A$ and $C$ commute one must have
\begin{displaymath}
C = \begin{pmatrix}
v & 0 \\ 
0 & v^{-1}
\end{pmatrix} \text{ for some } v \in \CC\setminus \lbrace 0,-1,1 \rbrace \; .
\end{displaymath}
One can then check that both equations $B = (AB)^{-1}B(AB)$ and $B = (CB) C (CB)^{-1}$ (in the variable $B$) admits a solution for any $u,v\neq 0, \pm 1$ and that such a triple $(A,B,C)$ defines a representation of $B_4$ for any such $u,v$. Therefore, the representation cannot be rigid.

\subsubsection{Groups of type $G(T)$ and $G_p(T)$}

Recall that since these groups are solvable, any of their irreducible representation into $PSL_2(\CC)$ must have dihedral image (see Proposition~\ref{prop:reprSolv}).

\paragraph{The group $G(t^3-1)$.}
Since $t^3-1=(t-1)(t^2+t+1)$, the group $G(t^3+1)$ is isomorphic to the semi--direct product $\ZZ^3 \,_\varphi\!\!\rtimes \ZZ$, where:
\begin{displaymath}
\varphi(n) \cdot \begin{pmatrix}
u \\ 
v \\ 
w
\end{pmatrix} = \begin{pmatrix}
0 & -1 & 0 \\ 
1 & -1 & 0 \\ 
0 & 0 & 1
\end{pmatrix}^n  \begin{pmatrix}
u \\ 
v \\ 
w
\end{pmatrix}\,.
\end{displaymath}
Fix the following set of generators for $G(t^3-1)$:
\begin{displaymath}
a:=\left( \begin{pmatrix}
1 \\ 
0 \\ 
0
\end{pmatrix} \, , \, 0 \right), \; b:=\left( \begin{pmatrix}
0 \\ 
1 \\ 
0
\end{pmatrix} \, , \, 0 \right), \;c:=\left( \begin{pmatrix}
0 \\ 
0 \\ 
1
\end{pmatrix} \, , \, 0 \right), \; d:=\left( \begin{pmatrix}
0 \\ 
0 \\ 
0
\end{pmatrix} \, , \, 1 \right) \; 
\end{displaymath}
and let $A,B,C,D$ be their images under $\rho$. Then $C$ centralises the whole image of $\rho$, which we assumed nonabelian therefore it must be equal to the identity in $PSL_2(\CC)$. Moreover, since $A$ and $B$ commute and $\mathrm{Im}(\rho)$ is dihedral and non--abelian one must have, up to global conjugacy:
\begin{displaymath}
A = \begin{pmatrix}
u & 0 \\ 
0 & u^{-1}
\end{pmatrix} , \quad B = \begin{pmatrix}
v & 0 \\ 
0 & v^{-1}
\end{pmatrix} \quad \text{and} \quad D = \begin{pmatrix}
0 & -1 \\ 
1 & 0
\end{pmatrix}
\end{displaymath}
or 
\begin{displaymath}
A =  B = \begin{pmatrix}
0 & -1 \\ 
1 & 0
\end{pmatrix} \quad \text{and} \quad D = \begin{pmatrix}
u & 0 \\ 
0 & u^{-1}
\end{pmatrix}
\end{displaymath}
for some $u,v \in \CC^*$. Note that $A$ must be equal to $B$ in the second case because those two matrices need to commute.

Using the above generators, one gets that in $G(t^3-1)$ one has the relation $d\cdot a = b \cdot d$ and so one must have $DA=BD$, which is only possible in the first case and if $uv = \pm 1$. But one also has $d^2 \cdot a = b \cdot d^2$ and so $A=B= I_2$ in $PSL_2(\CC)$, which is impossible as the image of $\rho$ must be non--abelian.

\paragraph{Groups of type $G_p(t+1)$.}

Let $p$ be some prime integer; then the group $G_p(t+1)$ is isomorphic to the semi--direct product $\FF_p \,_\varphi\!\!\rtimes \ZZ$ where $\phi(n)\cdot \bar{k} = (-1)^n \bar{k}$. As such, it has two generators $a:=(\bar{1},0)$ and $b:=(\bar{0},1)$ such that (using multiplicative notation) $a^p = 1$ and $ab^2 = b^2 a$. 

If we consider a representation $\rho : G_p(t+1) \rightarrow PSL_2(\CC)$ and set $(A,B) : = (\rho(a), \rho(b))$ then since $B^2$ commutes with every element of $\mathrm{Im}(\rho)$ we must have $A^p = B^2 = I_2$, therefore the aforementioned image must be isomorphic to the semi--direct product $\FF_p \rtimes \FF_2$ and so is finite.

\subsubsection{Group $\Ar^2(2,3,5)$}

If one sets $A := \rho(a), \, B := \rho(b)$ and $C:=\rho(c)$ then one gets $A^2 = B^3 = C^5 = ABC = I_2$ in $PSL_2(\CC)$. Lifting this to $SL_2(\CC)$, one can thereofore assume that $A^2 = \pm I_2, \, B^3 = (AB)^5 = -I_2$ (up to changing $B$ to $-B$) and so $\Tr(A) \in \lbrace -2,0,2 \rbrace, \, \Tr(B) = 2 \cos(\pi/3) = 1$ and $\Tr(C) = 2\cos(\pi/5)$ or $2\cos(3\pi/5)=-2\cos(2\pi/5)$.

The study of hypergeometric equations by Schwartz's yielded a list of all triples of matrices $(M_1,  M_2,M_3)$ in $SL_2(\CC)$ satisfying $M_1 M_2 M_3 = I_2$ such that the group $\langle M_1, M_2, M_3 \rangle \leq SL_2(\CC)$ is finite. An account of these works can be found in Chapter IX of~\cite{Uniformisation}, with the list itself being reproduced on p. 310. If $\Tr(A) = 2 \cos(\pi/2) = 0$ then the aforementioned Schwartz's list gives us that the image of $\rho$ is finite, isomorphic to $\mathfrak{A}_5$. Else, $A$ must be equal to $\pm I_2$ and so $BC = \pm I_2$: $\mathrm{Im}(\rho)$ must be abelian.

\subsubsection*{"Unusual" groups}

Several of the exceptional groups appearing in Degtyarev's list can be eliminated for algebraic reasons, as we show in this paragraph.

\paragraph{The group $\Gamma_4$}

Consider the following group:
\begin{displaymath}
\Gamma_4 := \langle a,b,c \, | \, aba=bab,\, cbc=bcb, \, a(bcb^{-1})a = (bcb^{-1})a(bcb^{-1}) \rangle \; .
\end{displaymath}
Let $\rho$ be a representation of the above group into $PSL_2(\CC)$ and set $A$ (resp. $B$, $C$) to be the images $\rho(a)$ (resp. $\rho(b)$, $\rho(c)$). First remark that if $[B,C] = I_2$ in $PSL_2(\CC)$ then the second relation in the presentation above becomes:
\begin{displaymath}
C^2B = C B^2, \quad \text{\emph{i.e}} \quad C=B 
\end{displaymath}
therefore $\rho$ factors through a representation of the group:
\begin{displaymath}
\langle a,b \, | \, aba=bab \rangle \;
\end{displaymath}
which is isomorphic to the braid group $B_3$ and so cannot satisfy both conditions \textbf{(C1)} and \textbf{(C2)}.

Lets assume then that $B$ does not commute to $C$. This implies that the restriction $\tilde{\rho}$ of $\rho$ to $\langle B,C \rangle$ factors through the braid group $B_3$, therefore it cannot have dihedral image by Lemma~\ref{lem:braidGroupDihedral}. So we get from Lemma~\ref{lem:braidGroupNonDihedral} that there is an analytic family of matrices $u\mapsto B(u),C(u)$ containing $B$ and $C$ such that $C(u)B(u)C(u)=B(u)C(u)B(u)$ for any $u$ in some Zariski--open set in $\CC$. Therefore, $[A,B]$ must not be the identity or else one would have $A=B$ and so $\rho$ would not be rigid. But then, one checks that, similarly to what we did in Lemma~\ref{lem:braidGroupNonDihedral}, the braid relations
\begin{displaymath}
AB(u)A=bab \quad \text{ and } \quad A(B(u)C(u)B(u)^{-1})A= (B(u)C(u)B(u)^{-1})A(B(u)C(u)B(u)^{-1})
\end{displaymath}
are compatible and so give us a way to analytically deform $A$ as well, and so $\rho$ is not rigid.

\paragraph{The group $\Gamma_3$}

We are looking at representations of the group
\begin{displaymath}
\Gamma_3 := \langle a,b \, | \, [a^3,b] = 1 \, , ab^2=ba^2 \rangle
\end{displaymath}
into $PSL_2(\CC)$. Let $\rho$ be such a representation and set $A$ (resp. $B$) to be the images $\rho(a)$ (resp. $\rho(b)$), then Lemma ~\ref{lem:centralPSL} forces $A^3$ to be the identity and so up to conjugacy it must lift to the following matrix in $SL_2(\CC)$:
\begin{displaymath}
\begin{pmatrix}
r  & 0 \\ 
0 & r^{-1}
\end{pmatrix}
\end{displaymath}
where $r$ is some root of the polynomial $Z^2+Z+1$. Then the relation $AB^2 = B^2 A$ and condition \textbf{(C1}) force $B$ to lift to
\begin{displaymath}
\begin{pmatrix}
\frac{r+2}{3}  & 1 \\ 
-\frac{2}{3} & \frac{1-r}{3}
\end{pmatrix} \; .
\end{displaymath}
However, the triple $(A,B,(AB)^{-1})$ appears in Schwartz's list~\cite{Uniformisation}; as a consequence, such a representation has finite image (in this case isomorphic to $\mathfrak{A}_4$).

\paragraph{The group $\Gamma_3^\prime$}

This group is given by the following presentation:
\begin{displaymath}
 \Gamma_3^\prime := \langle a,b,c \, | \, aca=cac,\, [b,c]=1, \, (ab)^2=(ba)^2 \rangle \; .
 \end{displaymath} 
 Let $\rho$ be a representation of the above group into $PSL_2(\CC)$ and set $A$ (resp. $B$, $C$) to be the images $\rho(a)$ (resp. $\rho(b)$, $\rho(c)$). Condition \textbf{(C1}) and the fact that $B$ and $C$ commute force the pair $(A,C)$ to be non--commutative. Since $A$ and $C$ are conjugate, it is possible to assume, up to conjugacy that:
\begin{displaymath}
A = \begin{pmatrix}
u & 1 \\ 
0 & u^{-1}
\end{pmatrix}, \quad B = \begin{pmatrix}
v & 0 \\ 
w & v^{-1}
\end{pmatrix}
\text{ and } C = \begin{pmatrix}
u & 0 \\ 
t & u^{-1}
\end{pmatrix} 
\end{displaymath} 
for some $u,v,w,t \in \CC^*$. Solving the equations $ACA=CAC$ and $(AB)^2 = (BA)^2$ in $PSL_2(\CC)$ in the aforementioned variables, we get the following one--parameter families ($u \in \CC^*$) of representations satisfying condition \textbf{(C1}):
\begin{enumerate}
\item
\begin{displaymath}
A= \begin{pmatrix}
u^{-1} & 1 \\ 
0 & u
\end{pmatrix} \quad B= \begin{pmatrix}
u^2 & 0 \\ 
-(u+u^{-1}) & u^{-2}
\end{pmatrix}\quad C= \begin{pmatrix}
u & 0 \\ 
-1 & u^{-1} 
\end{pmatrix}\; ;
\end{displaymath}
\item
\begin{displaymath}
A= \begin{pmatrix}
u & 1 \\ 
0 & u^{-1}
\end{pmatrix} \quad B= \begin{pmatrix}
u^2 & 0 \\ 
-\dfrac{(u^2+1)(u^4-u^2+1)}{u^3} & u^{-2}
\end{pmatrix}\quad C= \begin{pmatrix}
u & 0 \\ 
-\dfrac{u^4-u^2+1}{u^2} & u^{-1} 
\end{pmatrix}\; .
\end{displaymath}
\end{enumerate}
However, it is quite straightforward to check that these are not conjugate to representations with dihedral image: indeed if that were the case then in any pair of non--commutating matrices there should be one with eigenvalues $\pm i$. This can only happen here if $u$ is equal to $\pm i$ ; this cannot happen in case 1. as the associated representation has reducible image, nor in case $2$ since $B$ would be equal to the identity. Consequently, it follows from Theorem~\ref{th:LPT-CS} that all of the above factor through a curve.

\subsubsection{Product groups of $\ZZ$ and one of the above}

Let $\rho : \ZZ \times G \rightarrow PSL_2(\CC)$ be a group representation satisfying condition \textbf{(C1}) and set $a$ to be a generator of the "$\ZZ$--part of the above product. Then $\rho(a	)$ centralises the entire image of $\rho$, is a non--abelian subgroup of $PSL_2(\CC)$: therefore $\rho(a)$ is the identity element in $PSL_2(\CC)$ and so $\rho$ factorises through a representation of $G$. It follows that there is no representation of $\ZZ \times G$ satisfying conditions \textbf{(C1}) and \textbf{(C2}), as this would imply there exists one of $G$.

\subsubsection{Filtered list}

\begin{table}[!h]
\begin{center}
\begin{tabular}{ccc}
\textbf{Curve type} & \textbf{Intersection type(s)} &\textbf{Group(s)} \\ 
\hline 
$C_5(A_6 \sqcup 3A_2)$ & --- & $\Gamma_5$ \\ 
$C_3(A_1)\sqcup 2C_1$ &  $\underline{\times 3} \; ; \underline{\times 1}, \times 2$ & $G(t^2-1)$ \\ 
  & $\underline{\times 1}, \times 2 \; ; \underline{\times 1}, \times 2$ & $G(t^2-1)$ \\ 
$2C_2 \sqcup C_1$ & the two $C_2$ intersect with multiplicity $4$	 & $T_{2,4}$ \\ 
  & the two $C_2$ intersect at two points &  $T_{2,4}$ \\ 
$C_2 \sqcup 3C_1$ & the three $C_1$ have a common point & $\Gamma_2, \, \ZZ \times \Fb_2$ \\ 
  & else & $\Gamma_2^\prime, \,  \ZZ \times T_{2,4}$ \\ 
\end{tabular} 
\end{center}
\caption{Degtyarev's list, after elimination of groups.}\label{tab:Deg2}
\end{table}

Proposition~\ref{prop:DegGroups} allows us to substantially reduce Degtyarev's list, as evidenced in Table~\ref{tab:Deg2}. It now only remains to look at this new list on a curve by curve basis to conclude the proof of Theorem~\ref{thA:quintics}.

\subsection{Remaining quintic curves and their fundamental group}

\subsubsection{Irreducible quintics}

The only such curve with infinite non--abelian fundamental group is that of type $C_5(A_6 \sqcup 3A_2)$; the aforementioned group being isomorphic to 
\begin{displaymath}
\Gamma_5 := \langle u,v \, | u^3 = v^7 =(uv^2)^2 \rangle \; .
\end{displaymath}

This case has been previously studied by Cousin (see~\cite{theseGael} Section~5.2, pages~110--112): any such representation factors through $10:1$ ramified cover over $\PP^1$, and so cannot satisfy condition \textbf{(C2}).

\subsubsection{Curves of type $C_3(A_1) \sqcup 2 C1$}\label{sec:cubic}

We already eliminated the case where the intersection type is $\times 3, \times 3$ so we are left with $\underline{\times 3}; \underline{\times 1} \times 2$ and $\underline{\times 1} \times 2;\underline{\times 1} \times 2$. The fundamental group of the complement of both these curves is isomorphic to the solvable group $G(t^2 -1)$ so $\rho$ must be dihedral (see Proposition~\ref{prop:reprSolv}). This group is isomorphic to $\ZZ^2 \,_\varphi\!\!\rtimes \ZZ$, where:
\begin{displaymath}
\varphi(n) \cdot \begin{pmatrix}
u \\ 
v 
\end{pmatrix} = \begin{pmatrix}
0 & 1  \\ 
1 & 0 
\end{pmatrix}^n  \begin{pmatrix}
u \\ 
v 
\end{pmatrix}\, ,
\end{displaymath}
therefore it has three generators 
\begin{displaymath}
a := \left( \begin{pmatrix}
1 \\ 
0 
\end{pmatrix}\, , \, 0 \right) \; , \; b := \left( \begin{pmatrix}
0 \\ 
1 
\end{pmatrix}\, , \, 0 \right) \; \text{ and } \; c := \left( \begin{pmatrix}
0 \\ 
0 
\end{pmatrix}\, , \, 1 \right)
\end{displaymath}
such that  $b \cdot c^2=c^2 \cdot b$, $ca = bc$ and $[a,b]=1$. Thus if one sets $A,B,C$ to be the images of these generators in $PSL_2(\CC)$, then one has necessarily $C^2 = I_2$. Since the image of $\rho$ must be dihedral non--abelian, this implies that up to global conjugacy one gets diagonal $A$ and $B$ and:
\begin{displaymath}
C = \begin{pmatrix}
0 & 1 \\ 
-1 & 0
\end{pmatrix} \;.
\end{displaymath}
Therefore, the relation $CA=BC$ implies that there exists $u \in \CC^*$ such that:
\begin{displaymath}
A = \begin{pmatrix}
u & 0 \\ 
0 & u^{-1}
\end{pmatrix} , \quad \text{and} \quad B = \begin{pmatrix}
u^{-1} & 0 \\ 
0 & u
\end{pmatrix} \;.
\end{displaymath}
This is the third item in Theorem~\ref{thA:quintics}.

\subsubsection{Curves of type $2 C_2 \sqcup  C_1$}

We already treated the "$\Fb_2$--cases" in Degtyarev's list using Remark~\ref{rem:noCoverNeeded}, so we only need to concern ourselves with the remaining cases.
 $ $\\
$ $\\
$ $\\

\begin{figure}[!h]
\begin{center}
\begingroup%
  \makeatletter%
  \providecommand\color[2][]{%
    \errmessage{(Inkscape) Color is used for the text in Inkscape, but the package 'color.sty' is not loaded}%
    \renewcommand\color[2][]{}%
  }%
  \providecommand\transparent[1]{%
    \errmessage{(Inkscape) Transparency is used (non-zero) for the text in Inkscape, but the package 'transparent.sty' is not loaded}%
    \renewcommand\transparent[1]{}%
  }%
  \providecommand\rotatebox[2]{#2}%
  \ifx\svgwidth\undefined%
    \setlength{\unitlength}{338.21736852bp}%
    \ifx\svgscale\undefined%
      \relax%
    \else%
      \setlength{\unitlength}{\unitlength * \real{\svgscale}}%
    \fi%
  \else%
    \setlength{\unitlength}{\svgwidth}%
  \fi%
  \global\let\svgwidth\undefined%
  \global\let\svgscale\undefined%
  \makeatother%
  \begin{picture}(1,0.68377002)%
    \put(0,0){\includegraphics[width=\unitlength]{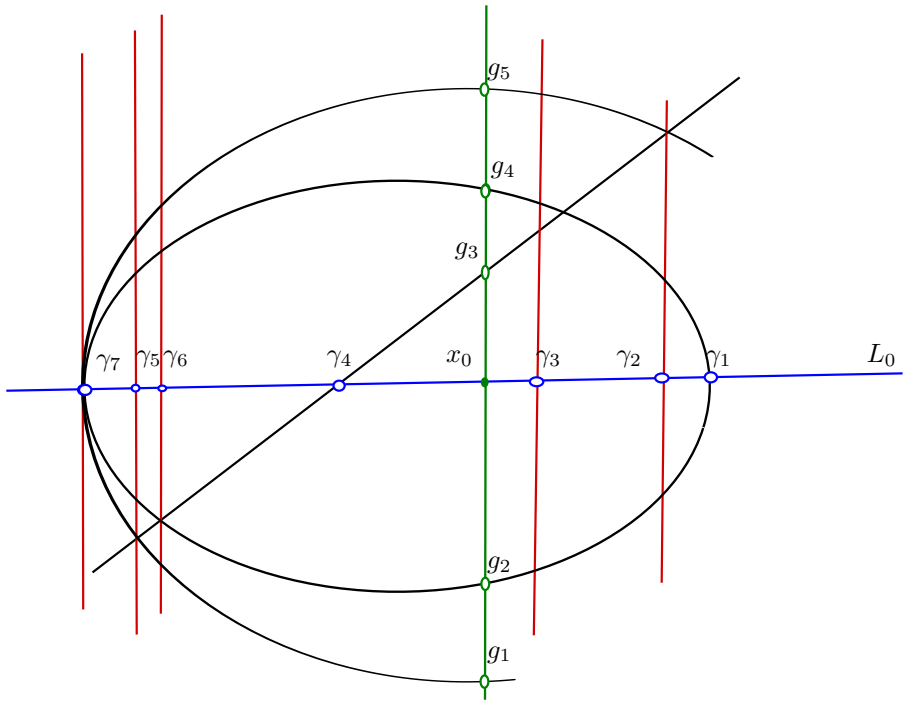}}%
    \put(0.77,0.38){\color[rgb]{0,0,0}\makebox(0,0)[lb]{\smash{ $\gamma_1$}}}%
    \put(0.68,0.38){\color[rgb]{0,0,0}\makebox(0,0)[lb]{\smash{$\gamma_2$}}}%
    \put(0.59,0.38){\color[rgb]{0,0,0}\makebox(0,0)[lb]{\smash{$\gamma_3$}}}%
    \put(0.35748418,0.38){\color[rgb]{0,0,0}\makebox(0,0)[lb]{\smash{$\gamma_4$}}}%
    \put(0.175,0.38){\color[rgb]{0,0,0}\makebox(0,0)[lb]{\smash{$\gamma_6$}}}%
    \put(0.145,0.38){\color[rgb]{0,0,0}\makebox(0,0)[lb]{\smash{$\gamma_5$}}}%
    \put(0.1,0.375){\color[rgb]{0,0,0}\makebox(0,0)[lb]{\smash{$\gamma_7$}}}%
    \put(0.54,0.59){\color[rgb]{0,0,0}\makebox(0,0)[lb]{\smash{$g_4$}}}%
    \put(0.50,0.5){\color[rgb]{0,0,0}\makebox(0,0)[lb]{\smash{$g_3$}}}%
    \put(0.536,0.7){\color[rgb]{0,0,0}\makebox(0,0)[lb]{\smash{$g_5$}}}%
    \put(0.536,0.15){\color[rgb]{0,0,0}\makebox(0,0)[lb]{\smash{$g_2$}}}%
    \put(0.536,0.05){\color[rgb]{0,0,0}\makebox(0,0)[lb]{\smash{$g_1$}}}%
    \put(0.49,0.38){\color[rgb]{0,0,0}\makebox(0,0)[lb]{\smash{$x_0$}}}%
    \put(0.94885004,0.38){\color[rgb]{0,0,0}\makebox(0,0)[lb]{\smash{ $L_0$}}}%
  \end{picture}%
\endgroup
\end{center}
\caption{First case. } \label{fig:ZVK_2C2_b}
\end{figure}

\begin{figure}[!h]
\begin{center}
\begingroup%
  \makeatletter%
  \providecommand\color[2][]{%
    \errmessage{(Inkscape) Color is used for the text in Inkscape, but the package 'color.sty' is not loaded}%
    \renewcommand\color[2][]{}%
  }%
  \providecommand\transparent[1]{%
    \errmessage{(Inkscape) Transparency is used (non-zero) for the text in Inkscape, but the package 'transparent.sty' is not loaded}%
    \renewcommand\transparent[1]{}%
  }%
  \providecommand\rotatebox[2]{#2}%
  \ifx\svgwidth\undefined%
    \setlength{\unitlength}{311.58935912bp}%
    \ifx\svgscale\undefined%
      \relax%
    \else%
      \setlength{\unitlength}{\unitlength * \real{\svgscale}}%
    \fi%
  \else%
    \setlength{\unitlength}{\svgwidth}%
  \fi%
  \global\let\svgwidth\undefined%
  \global\let\svgscale\undefined%
  \makeatother%
  \begin{picture}(1,0.81581655)%
    \put(0,0){\includegraphics[width=\unitlength]{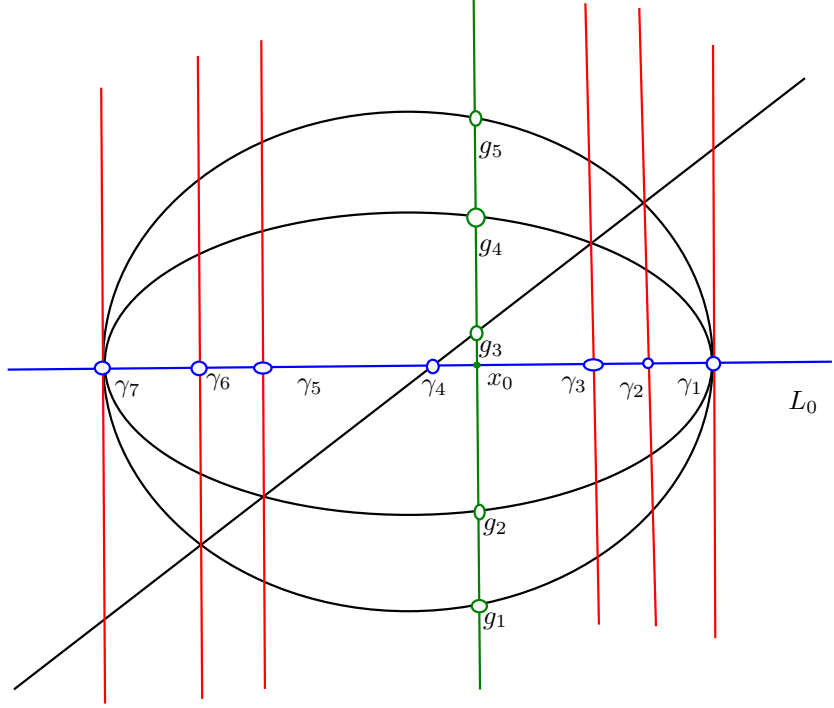}}%
    \put(0.551104,0.38300068){\color[rgb]{0,0,0}\makebox(0,0)[lb]{\smash{}}}%
    \put(0.55385483,0.38483462){\color[rgb]{0,0,0}\makebox(0,0)[lb]{\smash{}}}%
    \put(0.58,0.3885025){\color[rgb]{0,0,0}\makebox(0,0)[lb]{\smash{$x_0$}}}%
    \put(0.94447883,0.36){\color[rgb]{0,0,0}\makebox(0,0)[lb]{\smash{$L_0$}}}%
    \put(0.575,0.10){\color[rgb]{0,0,0}\makebox(0,0)[lb]{\smash{$g_1$}}}%
    \put(0.575,0.21336354){\color[rgb]{0,0,0}\makebox(0,0)[lb]{\smash{$g_2$}}}%
    \put(0.57,0.42884858){\color[rgb]{0,0,0}\makebox(0,0)[lb]{\smash{$g_3$}}}%
    \put(0.57,0.55080397){\color[rgb]{0,0,0}\makebox(0,0)[lb]{\smash{$g_4$}}}%
    \put(0.57,0.67){\color[rgb]{0,0,0}\makebox(0,0)[lb]{\smash{$g_5$}}}%
    \put(0.81,0.38300068){\color[rgb]{0,0,0}\makebox(0,0)[lb]{\smash{$\gamma_1$}}}%
    \put(0.74,0.37841592){\color[rgb]{0,0,0}\makebox(0,0)[lb]{\smash{$\gamma_2$}}}%
    \put(0.67,0.38483462){\color[rgb]{0,0,0}\makebox(0,0)[lb]{\smash{$\gamma_3$}}}%
    \put(0.50067126,0.38116674){\color[rgb]{0,0,0}\makebox(0,0)[lb]{\smash{$\gamma_4$}}}%
    \put(0.35,0.38300068){\color[rgb]{0,0,0}\makebox(0,0)[lb]{\smash{$\gamma_5$}}}%
    \put(0.24,0.3885025){\color[rgb]{0,0,0}\makebox(0,0)[lb]{\smash{$\gamma_6$}}}%
    \put(0.13,0.38){\color[rgb]{0,0,0}\makebox(0,0)[lb]{\smash{$\gamma_7$}}}%
  \end{picture}%
\endgroup
\end{center}
\caption{Second case. } \label{fig:ZVK_2C2_a}
\end{figure}

First, we consider the case where the two quadric curves have an intersection point of multiplicity $4$ and the linear component is not the common tangent at the aforementioned point; an example of such a pair of curves is given (in homogeneous coordinates $[x:y:z]$ by the equation 
\begin{displaymath}
(yz-x^2)(yz-x^2 - y^2) = 0 \; .
\end{displaymath}
It is a straightforward application of the Zariski--Van Kampen  method (see Section 1.1.2 in \cite{AThese} or \cite{LNZV} for an account of this method) to show that the local monodromy around the linear component centralises the whole fundamental group. Indeed, if one looks at the loops given inFigure~\ref{fig:ZVK_2C2_b} then the braid monodromy relations yield that the fundamental group of the complement of this curve is generated by the loops $g_1\,,g_2$ and $g_3$ and that $g_3$ centralises the other two. Therefore, no such representation  can satisfy conditions \textbf{(C1}) and \textbf{(C2}).

The case where the two conics intersect at two points is treated in much the same way (see Figure~\ref{fig:ZVK_2C2_a}); here again the local monodromy around the line, materialised by the loop $g_3$ must centralise the entire group and so be trivial, which contradicts conditions \textbf{(C1}) and \textbf{(C2}).

\subsubsection{Curves of type $C_2\sqcup 3 C_1$}
\begin{figure}[!h]
\begin{center}
\begingroup%
  \makeatletter%
  \providecommand\color[2][]{%
    \errmessage{(Inkscape) Color is used for the text in Inkscape, but the package 'color.sty' is not loaded}%
    \renewcommand\color[2][]{}%
  }%
  \providecommand\transparent[1]{%
    \errmessage{(Inkscape) Transparency is used (non-zero) for the text in Inkscape, but the package 'transparent.sty' is not loaded}%
    \renewcommand\transparent[1]{}%
  }%
  \providecommand\rotatebox[2]{#2}%
  \ifx\svgwidth\undefined%
    \setlength{\unitlength}{350.30528679bp}%
    \ifx\svgscale\undefined%
      \relax%
    \else%
      \setlength{\unitlength}{\unitlength * \real{\svgscale}}%
    \fi%
  \else%
    \setlength{\unitlength}{\svgwidth}%
  \fi%
  \global\let\svgwidth\undefined%
  \global\let\svgscale\undefined%
  \makeatother%
  \begin{picture}(1,0.48869902)%
    \put(0,0){\includegraphics[width=\unitlength]{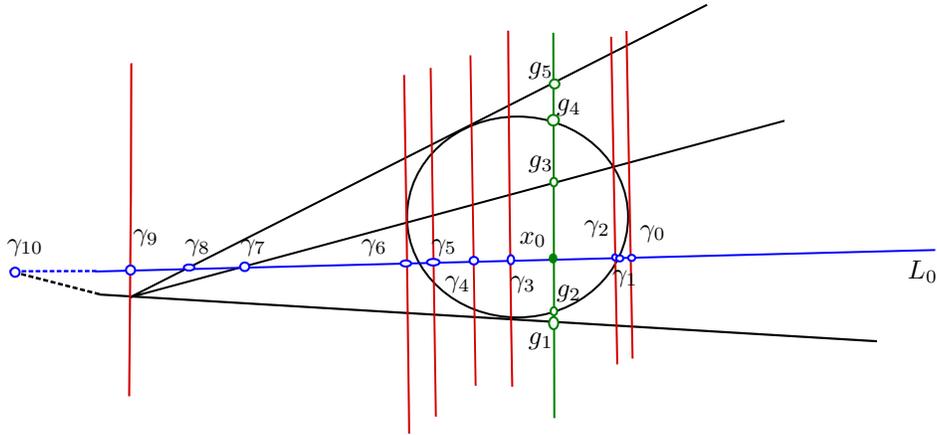}}%
    \put(0.96762661,0.17){\color[rgb]{0,0,0}\makebox(0,0)[lb]{\smash{$L_0$}}}%
    \put(0.65,0.167){\color[rgb]{0,0,0}\makebox(0,0)[lb]{\smash{$\gamma_1$}}}%
    \put(0.62,0.22){\color[rgb]{0,0,0}\makebox(0,0)[lb]{\smash{$\gamma_2$}}}%
    \put(0.54,0.16){\color[rgb]{0,0,0}\makebox(0,0)[lb]{\smash{$\gamma_3$}}}%
    \put(0.55,0.20772693){\color[rgb]{0,0,0}\makebox(0,0)[lb]{\smash{$x_0$}}}%
    \put(0.47,0.16){\color[rgb]{0,0,0}\makebox(0,0)[lb]{\smash{$\gamma_4$}}}%
    \put(0.455,0.2){\color[rgb]{0,0,0}\makebox(0,0)[lb]{\smash{$\gamma_5$}}}%
    \put(0.38,0.2){\color[rgb]{0,0,0}\makebox(0,0)[lb]{\smash{$\gamma_6$}}}%
    \put(0.25,0.2){\color[rgb]{0,0,0}\makebox(0,0)[lb]{\smash{$\gamma_7$}}}%
    \put(0.13456521,0.21){\color[rgb]{0,0,0}\makebox(0,0)[lb]{\smash{$\gamma_9$}}}%
    \put(0.19,0.2){\color[rgb]{0,0,0}\makebox(0,0)[lb]{\smash{$\gamma_8$}}}%
    \put(0.56,0.099){\color[rgb]{0,0,0}\makebox(0,0)[lb]{\smash{$g_1$}}}%
    \put(0.59,0.15){\color[rgb]{0,0,0}\makebox(0,0)[lb]{\smash{$g_2$}}}%
    \put(0.56,0.29){\color[rgb]{0,0,0}\makebox(0,0)[lb]{\smash{$g_3$}}}%
    \put(0.59,0.35){\color[rgb]{0,0,0}\makebox(0,0)[lb]{\smash{$g_4$}}}%
    \put(0.56,0.39){\color[rgb]{0,0,0}\makebox(0,0)[lb]{\smash{$g_5$}}}%
    \put(0.68,0.215){\color[rgb]{0,0,0}\makebox(0,0)[lb]{\smash{$\gamma_0$}}}%
     \put(-0.00089477,0.2){\color[rgb]{0,0,0}\makebox(0,0)[lb]{\smash{$\gamma_{10}$}}}%
  \end{picture}%
\endgroup
\end{center}
\caption{First case. } \label{fig:ZVK_3C2_c}
\end{figure}
We already eliminated one of those when looking at triple singularities; let us now deal with the others.

\paragraph{Case 1: the three lines have a common point.} If two of them are tangent to the conic then the fundamental group of the complement of this curve is given by:
\begin{displaymath}
\Gamma_2 := \langle a,b,c \, | \,  [a,b]=[a,c^{-1}bc] = 1, \, (bc)^2=(cb)^2 \rangle. 
\end{displaymath}
More precisely, if one applies the Zariski--Van Kampen method to this curve, one can identify $a,\,b$ and $c$ with (respectively) the loops $g_3, \,g_2$ and $g_1$ in Figure~\ref{fig:ZVK_3C2_c}.

If one sets $A:=\rho(a)$, $B:= \rho(b)$ and $C:=\rho(c)$ then it follows from the fact that the local monodromy must be non--trivial and Lemma~\ref{lem:centralPSL} that $B$ must commute to $C^{-1}BC$. Thus, the commutator $[B,C^2]$ must be trivial, hence $C^2$ must be the identity element in $PSL_2(\CC)$ and since $C$ cannot be trivial and $[A,B] = 1$ then up to global conjugacy $\rho$ must be of the following type:
\begin{displaymath}
\rho : a  \mapsto \begin{pmatrix}
u & 0 \\ 
0 & u^{-1}
\end{pmatrix}  , \quad 
b  \mapsto \begin{pmatrix}
v & 0 \\ 
0 & v^{-1}
\end{pmatrix}, \quad
c  \mapsto \begin{pmatrix}
0 & 1 \\ 
-1 & 0
\end{pmatrix} \; , \text{ for }  u,v \in \CC^* \; .
\end{displaymath}
This is the second case in Theorem~\ref{thA:quintics}.

\begin{figure}[!h]
\begin{center}
\begingroup%
  \makeatletter%
  \providecommand\color[2][]{%
    \errmessage{(Inkscape) Color is used for the text in Inkscape, but the package 'color.sty' is not loaded}%
    \renewcommand\color[2][]{}%
  }%
  \providecommand\transparent[1]{%
    \errmessage{(Inkscape) Transparency is used (non-zero) for the text in Inkscape, but the package 'transparent.sty' is not loaded}%
    \renewcommand\transparent[1]{}%
  }%
  \providecommand\rotatebox[2]{#2}%
  \ifx\svgwidth\undefined%
    \setlength{\unitlength}{356.25562786bp}%
    \ifx\svgscale\undefined%
      \relax%
    \else%
      \setlength{\unitlength}{\unitlength * \real{\svgscale}}%
    \fi%
  \else%
    \setlength{\unitlength}{\svgwidth}%
  \fi%
  \global\let\svgwidth\undefined%
  \global\let\svgscale\undefined%
  \makeatother%
  \begin{picture}(1,0.48053655)%
    \put(0,0){\includegraphics[width=\unitlength]{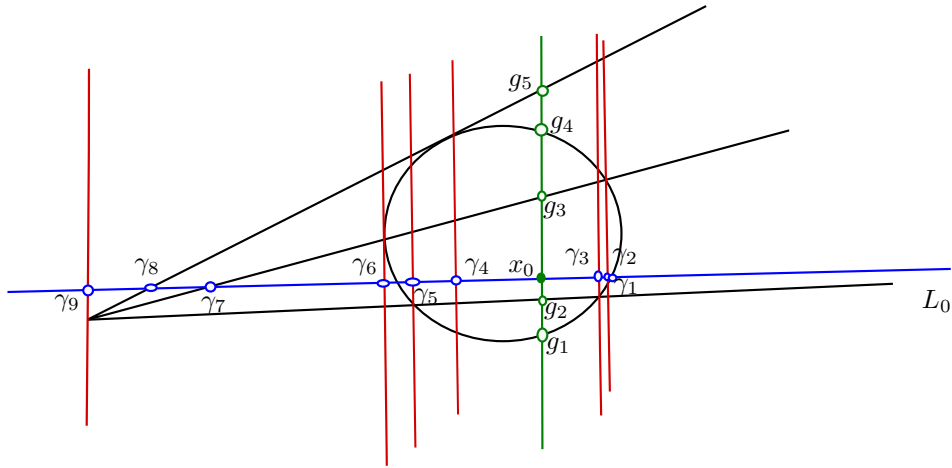}}%
    \put(0.96762661,0.17){\color[rgb]{0,0,0}\makebox(0,0)[lb]{\smash{$L_0$}}}%
    \put(0.64154747,0.18855861){\color[rgb]{0,0,0}\makebox(0,0)[lb]{\smash{$\gamma_1$}}}%
    \put(0.64,0.22){\color[rgb]{0,0,0}\makebox(0,0)[lb]{\smash{$\gamma_2$}}}%
    \put(0.59,0.22){\color[rgb]{0,0,0}\makebox(0,0)[lb]{\smash{$\gamma_3$}}}%
    \put(0.53,0.20772693){\color[rgb]{0,0,0}\makebox(0,0)[lb]{\smash{$x_0$}}}%
    \put(0.48389528,0.21){\color[rgb]{0,0,0}\makebox(0,0)[lb]{\smash{$\gamma_4$}}}%
    \put(0.43002131,0.178){\color[rgb]{0,0,0}\makebox(0,0)[lb]{\smash{$\gamma_5$}}}%
    \put(0.365,0.21){\color[rgb]{0,0,0}\makebox(0,0)[lb]{\smash{$\gamma_6$}}}%
    \put(0.20488489,0.17){\color[rgb]{0,0,0}\makebox(0,0)[lb]{\smash{$\gamma_7$}}}%
    \put(0.13456521,0.21){\color[rgb]{0,0,0}\makebox(0,0)[lb]{\smash{$\gamma_8$}}}%
    \put(0.05,0.17154579){\color[rgb]{0,0,0}\makebox(0,0)[lb]{\smash{$\gamma_9$}}}%
    \put(0.57,0.12617822){\color[rgb]{0,0,0}\makebox(0,0)[lb]{\smash{$g_1$}}}%
    \put(0.56782522,0.16303938){\color[rgb]{0,0,0}\makebox(0,0)[lb]{\smash{$g_2$}}}%
    \put(0.56669099,0.27078728){\color[rgb]{0,0,0}\makebox(0,0)[lb]{\smash{$g_3$}}}%
    \put(0.57406327,0.36){\color[rgb]{0,0,0}\makebox(0,0)[lb]{\smash{$g_4$}}}%
    \put(0.53039702,0.405){\color[rgb]{0,0,0}\makebox(0,0)[lb]{\smash{$g_5$}}}%
  \end{picture}%
\endgroup
\end{center}
\caption{Second case. } \label{fig:ZVK_3C2_a}
\end{figure}

On the other hand, if at most one of the lines is tangent to the conic, we derive from the Zariski--Van Kampen method that the local monodromy around one of the non--tangent lines must centralise the whole fundamental group. More precisely, braid monodromy relations from Figure~\ref{fig:ZVK_3C2_a} yield that the loop $g_3$ is in the centraliser of the fundamental group of the complement of the pictured curve.
\begin{center}
\begin{figure}

\scalebox{1} 
{
\begin{pspicture}(0,-3.85)(17.841875,3.85)
\definecolor{color3}{rgb}{0.00392156862745098,0.00392156862745098,0.00392156862745098}
\definecolor{color18}{rgb}{0.0196078431372549,0.611764705882353,0.00392156862745098}
\definecolor{color36}{rgb}{0.0392156862745098,0.9529411764705882,0.09803921568627451}
\definecolor{color20}{rgb}{0.19215686274509805,0.6274509803921569,0.00392156862745098}
\pscircle[linewidth=0.04,dimen=outer](6.53,-0.28){1.79}
\psline[linewidth=0.04cm,linecolor=color3](0.16,-3.83)(7.66,3.23)
\psline[linewidth=0.04cm,linecolor=color3](4.32,3.83)(11.12,-1.71)
\psline[linewidth=0.04cm,linecolor=color3](0.0,-3.47)(13.8,-0.57)
\psline[linewidth=0.04cm,linecolor=color18](14.3,-0.65)(1.44,-0.85)

\usefont{T1}{ptm}{m}{n}
\rput(11.3142185,-0.445){\color{color36}$L$}
\pscustom[linewidth=0.04]
{
\newpath
\moveto(9.02,-0.73)
\lineto(8.96,-0.83)
\curveto(8.93,-0.88)(8.88,-0.975)(8.86,-1.02)
\curveto(8.84,-1.065)(8.795,-1.17)(8.77,-1.23)
\curveto(8.745,-1.29)(8.705,-1.39)(8.69,-1.43)
\curveto(8.675,-1.47)(8.65,-1.56)(8.64,-1.61)
\curveto(8.63,-1.66)(8.62,-1.75)(8.62,-1.79)
\curveto(8.62,-1.83)(8.62,-1.915)(8.62,-1.96)
\curveto(8.62,-2.005)(8.635,-2.085)(8.65,-2.12)
\curveto(8.665,-2.155)(8.71,-2.2)(8.74,-2.21)
\curveto(8.77,-2.22)(8.825,-2.23)(8.85,-2.23)
\curveto(8.875,-2.23)(8.92,-2.225)(8.94,-2.22)
\curveto(8.96,-2.215)(9.0,-2.19)(9.02,-2.17)
\curveto(9.04,-2.15)(9.075,-2.105)(9.09,-2.08)
\curveto(9.105,-2.055)(9.13,-1.995)(9.14,-1.96)
\curveto(9.15,-1.925)(9.16,-1.86)(9.16,-1.83)
\curveto(9.16,-1.8)(9.155,-1.75)(9.15,-1.73)
\curveto(9.145,-1.71)(9.12,-1.66)(9.1,-1.63)
}
\pscustom[linewidth=0.04]
{
\newpath
\moveto(9.02,-1.51)
\lineto(8.95,-1.47)
\curveto(8.915,-1.45)(8.86,-1.425)(8.84,-1.42)
\curveto(8.82,-1.415)(8.785,-1.4)(8.77,-1.39)
}
\pscustom[linewidth=0.04]
{
\newpath
\moveto(9.02,-0.71)
\lineto(9.04,-0.63)
\curveto(9.05,-0.59)(9.06,-0.52)(9.06,-0.49)
\curveto(9.06,-0.46)(9.065,-0.395)(9.07,-0.36)
\curveto(9.075,-0.325)(9.085,-0.27)(9.09,-0.25)
\curveto(9.095,-0.23)(9.105,-0.19)(9.11,-0.17)
\curveto(9.115,-0.15)(9.115,-0.105)(9.11,-0.08)
\curveto(9.105,-0.055)(9.105,0.0)(9.11,0.03)
\curveto(9.115,0.06)(9.125,0.11)(9.13,0.13)
\curveto(9.135,0.15)(9.145,0.19)(9.15,0.21)
\curveto(9.155,0.23)(9.18,0.26)(9.2,0.27)
\curveto(9.22,0.28)(9.265,0.295)(9.29,0.3)
\curveto(9.315,0.305)(9.355,0.31)(9.37,0.31)
\curveto(9.385,0.31)(9.425,0.305)(9.45,0.3)
\curveto(9.475,0.295)(9.52,0.27)(9.54,0.25)
\curveto(9.56,0.23)(9.6,0.195)(9.62,0.18)
\curveto(9.64,0.165)(9.675,0.13)(9.69,0.11)
\curveto(9.705,0.09)(9.715,0.045)(9.71,0.02)
\curveto(9.705,-0.005)(9.695,-0.06)(9.69,-0.09)
\curveto(9.685,-0.12)(9.665,-0.17)(9.65,-0.19)
\curveto(9.635,-0.21)(9.605,-0.24)(9.59,-0.25)
\curveto(9.575,-0.26)(9.545,-0.28)(9.53,-0.29)
}
\pscustom[linewidth=0.04]
{
\newpath
\moveto(9.36,-0.37)
\lineto(9.29,-0.38)
\curveto(9.255,-0.385)(9.195,-0.395)(9.17,-0.4)
\curveto(9.145,-0.405)(9.09,-0.42)(9.06,-0.43)
}
\pscustom[linewidth=0.04]
{
\newpath
\moveto(9.02,-0.71)
\lineto(8.95,-0.67)
\curveto(8.915,-0.65)(8.86,-0.61)(8.84,-0.59)
\curveto(8.82,-0.57)(8.78,-0.53)(8.76,-0.51)
\curveto(8.74,-0.49)(8.695,-0.455)(8.67,-0.44)
\curveto(8.645,-0.425)(8.59,-0.39)(8.56,-0.37)
\curveto(8.53,-0.35)(8.485,-0.325)(8.47,-0.32)
\curveto(8.455,-0.315)(8.425,-0.3)(8.41,-0.29)
\curveto(8.395,-0.28)(8.355,-0.27)(8.33,-0.27)
\curveto(8.305,-0.27)(8.255,-0.27)(8.23,-0.27)
\curveto(8.205,-0.27)(8.16,-0.28)(8.14,-0.29)
\curveto(8.12,-0.3)(8.09,-0.325)(8.08,-0.34)
\curveto(8.07,-0.355)(8.05,-0.39)(8.04,-0.41)
\curveto(8.03,-0.43)(8.035,-0.465)(8.05,-0.48)
\curveto(8.065,-0.495)(8.105,-0.525)(8.13,-0.54)
\curveto(8.155,-0.555)(8.22,-0.58)(8.26,-0.59)
}
\pscustom[linewidth=0.04]
{
\newpath
\moveto(8.32,-0.61)
\lineto(8.37,-0.61)
\curveto(8.395,-0.61)(8.445,-0.605)(8.47,-0.6)
\curveto(8.495,-0.595)(8.535,-0.58)(8.55,-0.57)
\curveto(8.565,-0.56)(8.595,-0.54)(8.61,-0.53)
\curveto(8.625,-0.52)(8.67,-0.5)(8.7,-0.49)
}
\psdots[dotsize=0.12,linecolor=color3,fillstyle=solid,dotstyle=o](9.9,-0.71)
\psdots[dotsize=0.12,linecolor=color3,fillstyle=solid,dotstyle=o](8.22,-0.73)
\psdots[dotsize=0.12,linecolor=color3,fillstyle=solid,dotstyle=o](4.84,-0.79)
\psdots[dotsize=0.12,linecolor=color3,fillstyle=solid,dotstyle=o](3.36,-0.83)
\usefont{T1}{ptm}{m}{n}
\rput(8.58422,-0.125){$a$}
\usefont{T1}{ptm}{m}{n}
\rput(9.394218,-1.765){$b$}
\usefont{T1}{ptm}{m}{n}
\rput(9.944219,0.175){$c$}

\psdots[dotsize=0.12,linecolor=color3,fillstyle=solid,dotstyle=o](13.36,-0.67)

\end{pspicture} 
}

\caption{Third case}
\label{fig:ZVK_3C2_b}
\end{figure}
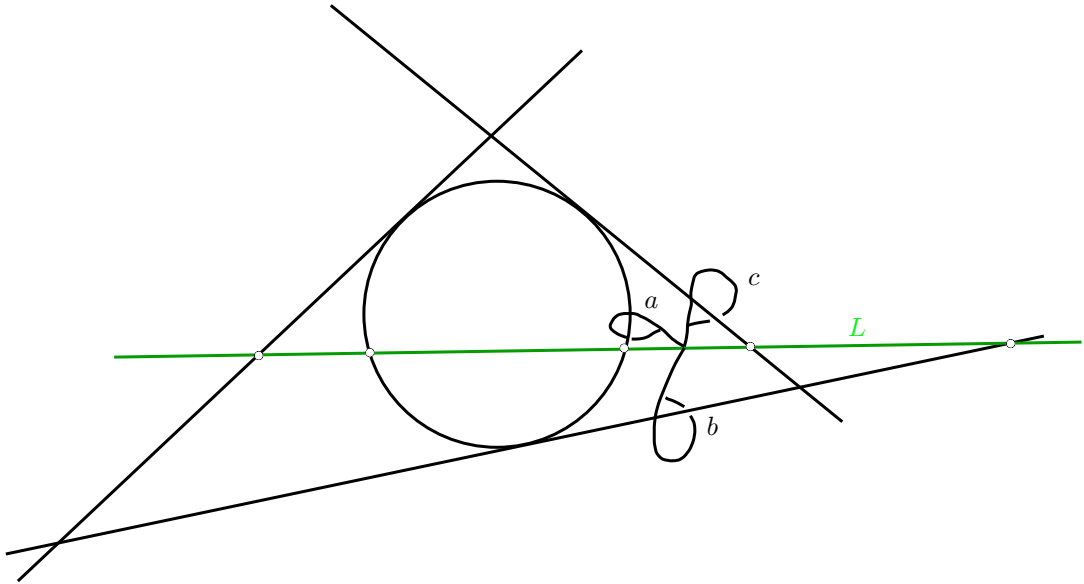

\end{center}

\paragraph{Case 2: the three lines do not have a common point.} In this case one must have at least two lines tangent to the conic in order to get a non--abelian fundamental group. If all three lines are tangent to the conic then the fundamental group of the complement is isomorphic to 
\begin{displaymath}
\Gamma_2^\prime := \langle a,b,c \, | \, (ab)^2 (ba)^{-2} = (ac)^2 (ca)^{-2}= [b,c]=1 \rangle \;  ,
\end{displaymath}
which we extensively studied in \cite{A2} and corresponds to the first case in Theorem~\ref{thA:quintics}. Otherwise, it is once again a consequence of braid monodromy relations that the local monodromy about any line not tangent to the conic must commute to every loop in the fundamental group (see Figure~\ref{fig:ZVK_3C2_b}), thus we can discard this type of curve as well.

\section{Mapping class group orbits}\label{sec:MCGOrbits}
We saw in Section \ref{sec:Quintics} that there are at most four families of "interesting" representations of fundamental group of complement of quintics into $PSL_2(\CC)$. The aim of this paper is to explain how one can use this to obtain isomonodromic deformations of the five punctured sphere and to describe the associated mapping class group orbits.
\subsection{General method}

Let $\nabla$ be a rank two logarithmic flat $\mathfrak{sl}_2(\CC)$--connection over $\PP^2(\CC)$ whose polar locus is exactly some quintic plane curve $Q \subset \PP^2(\CC)$ and let $\rho : \pi_1(\PP^2(\CC)\setminus Q) \rightarrow SL_2(\CC)$ be its monodromy representation. Assume that the representation $\rho$ is non--degenerate in the following sense.

\begin{defi}\label{def:nonDeg}
We say that the monodromy representation associated with a rank two logarithmic flat $\mathfrak{sl}_2(\CC)$--connection over $\PP^2-\Qr$ is non--degenerate if
\begin{itemize}
\item its image forms an irreducible subgroup of $SL_2(\CC)$ ; 
\item its local monodromy around any irreducible component of $\Qr$ is projectively non--trivial (i.e is non-trivial in $PSL_2(\CC)$).
\end{itemize}
\end{defi}

If the representation $\rho$ factors through an orbicurve, then it is a known fact~\cite{CorSim,LTP} that $\rho$ can be obtained as the monodromy of the pullback of some logarithmic flat connection over a curve. Isomonodromic deformations of punctured spheres arising from such constructions have been extensively studied by Diarra~\cite{Diar1,Diar2}, so let us assume that $\rho$ does not factor through an orbicurve. Therefore, it follows from Theorem~\ref{thA:quintics} that the pair $(\pi_1(\PP^2(\CC)\setminus Q),P\circ\rho)$ (where $P : SL_2(\CC) \rightarrow PSL_2(\CC)$ is the canonical projection) must be one of the following:
\begin{enumerate}
\item $\pi_1(\PP^2(\CC)\setminus Q) \cong \langle a,b,c \, | \, (ab)^2 (ba)^{-2} = (ac)^2 (ca)^{-2}= [b,c]=1 \rangle \;  $,
$$\rho : a  \mapsto \begin{pmatrix}
0 & 1 \\ 
-1 & 0
\end{pmatrix}  , \quad 
b  \mapsto \begin{pmatrix}
u & 0 \\ 
0 & u^{-1}
\end{pmatrix}, \quad
c  \mapsto \begin{pmatrix}
v & 0 \\ 
0 & v^{-1}
\end{pmatrix} \; , \text{ for some }  u,v \in \CC^* \; ;$$
\item $\pi_1(\PP^2(\CC)\setminus Q) \cong \langle a,b,c \, | \, [a,b] = [a,c^{-1}bc] = 1, \, (bc)^2 = (cb)^2\rangle$,
\begin{displaymath}
\rho : a  \mapsto \begin{pmatrix}
u & 0 \\ 
0 & u^{-1}
\end{pmatrix}  , \quad 
b  \mapsto \begin{pmatrix}
v & 0 \\ 
0 & v^{-1}
\end{pmatrix}, \quad
c  \mapsto \begin{pmatrix}
0 & 1 \\ 
-1 & 0
\end{pmatrix} \; , \text{ for some }  u,v \in \CC^* \; ;
\end{displaymath}
\item $\pi_1(\PP^2(\CC)\setminus Q) \cong \ZZ^2 \rtimes \ZZ$,
\begin{displaymath}
\rho : a  \mapsto \begin{pmatrix}
u &  0 \\ 
0 & u^{-1}
\end{pmatrix}  , \quad 
b  \mapsto \begin{pmatrix}
u^{-1} & 0 \\ 
0 & u
\end{pmatrix}, \quad
c  \mapsto \begin{pmatrix}
0 & 1 \\ 
-1 & 0
\end{pmatrix} \; , \text{ for some }  u \in \CC^* \; .
\end{displaymath}
\end{enumerate}

Let $L$ be a generic line in the projective plane $\PP^2(\CC)$ ; then $L$ must intersect the quintic $Q$ at exactly five points; identify $L$ to $\PP^1$ choosing a coordinate so that these are $0,1,\infty$ and some $t_1, t_2$. By restricting $\nabla$ to $L$, we get a logarithmic flat connection $\nabla_L$ over the punctured Riemann sphere $\PP^1_5 := \PP^1(\CC) \setminus \lbrace 0,1,t_1,t_2,\infty \rbrace$ whose monodromy $\rho_L$ is defined through the following diagram: 
\begin{displaymath}
\xymatrix{
\pi_1(\PP^1_5) \cong \Fb_4  \ar[rd]_{\rho_L}  \ar[r]^{\tau}& \pi_1(\PP^2(\CC) - Q) \ar[d]^{\rho}\\
 & SL_2(\CC) 
}
\end{displaymath}
where $\tau$ is the natural morphism given by restriction to $L$ ; the Lefschetz hyperplane theorem (see~\cite{Milnor}, Theorem 7.4) shows that it is in fact onto. By construction, there exists a Zariski--open subset $U$ of the dual projective space $\widehat{\PP^2(\CC)}$ such that the family $(\nabla_L)_{L \in U}$ is an isomorphic deformation of the five punctured sphere. 

\paragraph{Local monodromy}

Using the Zariski--Van Kampen method, it is actually quite straightforward to explicitly describe $\rho_L$ in the three cases above. Indeed, if one denotes $\Fb_4:=\langle d_1, \ldots, d_4 \, | \, \emptyset \rangle$ then the Lefschetz morphism $\tau$ is given by (respectively):

$$\begin{array}{ccllcll}
\mathrm{1.} & d_1 & \mapsto  b & \qquad &\mathrm{2.} & d_1 & \mapsto  c  \\ 
  & d_2 & \mapsto a& \qquad &   & d_2 & \mapsto b \\ 
  & d_3 & \mapsto b a b^{-1} & \qquad &   & d_3 & \mapsto a \\ 
  &d_4 & \mapsto c& \qquad &  & d_4 & \mapsto b \\ 
  &  &  &  & &   &   \\ 
\mathrm{3. \quad  (a)}  & d_1 & \mapsto  b & & \mathrm{3. \quad (b)} & d_1 & \mapsto  b \\ 
  & d_2 & \mapsto b a & &&d_2 & \mapsto a \\ 
  & d_3 & \mapsto a  && &d_3 & \mapsto a \\ 
  & d_4 & \mapsto  b^{-1} a b & & & d_4 & \mapsto  b^{-1} a b  \\ 
\end{array} $$

Note that each of the images of the $d_i$ correspond to the local monodromy around some irreducible component of the polar locus, in the following sense: let $C$ be an irreducible curve contained in the polar locus of some logarithmic flat $\mathfrak{sl}_2(\CC)$--connection $\nabla$ over $\PP^2$, with associated monodromy representation $\varrho$. Set a point $p \in C$ such that no other irreducible curve in the polar locus of the connection passes through $p$; then if $U$ is a sufficiently small analytic neighbourhood of $p$ one gets:
\begin{displaymath}
\pi_1(U \setminus C \cap U) \cong \ZZ \;.
\end{displaymath}
Let $\gamma$ be any loop generating the above cyclic group; the conjugacy class of the matrix $\varrho(\gamma)$ does not depend on the choice of a base point for the fundamental group. Indeed, if $\gamma$ is chosen as above for some base point $q$ and if $q^\prime$ is some other point in the complement of the polar locus, then if one takes $\delta$ to be any path between $q^\prime$ and $q$, the loop $\delta \cdot \gamma \cdot \delta^{-1}$ is an element of the fundamental group of the complement based at $q^\prime$ whose monodromy is conjugate to $\varrho (\gamma)$.

\begin{defi}\label{def:locMon}
Using the notations above, define the local monodromy of $\nabla$ around $C$ as the conjugacy class of the matrix $\varrho (\gamma)$.
\end{defi}

As such, the restricted monodromy $\rho_L$ must be given (up to global conjugacy) by the matrices appearing in Table~\ref{tab:monodromyLinesQuint}.

\begin{table}[!!h]
\begin{center}
\begin{tabular}{c|c|c|c|c|c}
Case & $x=0$ & $x=1$ & $x=t_1$ & $x=t_2$ & $x=\infty$ \\ 
\hline 
& & & & & \\
1.&$\begin{pmatrix}
v& 0 \\ 
0 & v^{-1}
\end{pmatrix}$ & $\begin{pmatrix}
u & 0 \\ 
0 &  u^{-1}
\end{pmatrix}$ & $\begin{pmatrix}
0 & 1 \\ 
-1 & 0
\end{pmatrix}$ & $\begin{pmatrix}
0 & u^{2}\\ 
-u^{-2} & 0
\end{pmatrix}$ & $\begin{pmatrix}
-uv^{-1} & 0 \\ 
0 & -u^{-1} v
\end{pmatrix}$ \\ 
& & & & \\
\hline 
& & & & & \\
2.&$\begin{pmatrix}
0& 1 \\ 
-1 &0
\end{pmatrix}$ & $\begin{pmatrix}
v & 0 \\ 
0 &  v^{-1}
\end{pmatrix}$ & $\begin{pmatrix}
u & 0 \\ 
0 &u^{-1}
\end{pmatrix}$ & $\begin{pmatrix}
v & 0 \\ 
0 &  v^{-1}
\end{pmatrix}$ & $\begin{pmatrix}
0 & -(uv^2)^{-1} \\ 
uv^2& 0 
\end{pmatrix}$ \\ 
& & & & \\
\hline 
& & & & & \\
3.  (a)&$\begin{pmatrix}
0& 1 \\ 
-1 &0
\end{pmatrix}$ & $\begin{pmatrix}
0& u^{-1} \\ 
-u &0
\end{pmatrix}$  & $\begin{pmatrix}
u & 0 \\ 
0 &u^{-1}
\end{pmatrix}$ & $\begin{pmatrix}
u^{-1} & 0 \\ 
0 & u
\end{pmatrix}$ & $\begin{pmatrix}
-u & 0 \\ 
0 &-u^{-1}
\end{pmatrix}$ \\ 
& & & & \\
\hline 
& & & & & \\
3.  (b)&$\begin{pmatrix}
0& -1 \\ 
1 &0
\end{pmatrix}$ &  $\begin{pmatrix}
u & 0 \\ 
0 &u^{-1}
\end{pmatrix}$ & $\begin{pmatrix}
u & 0 \\ 
0 &u^{-1}
\end{pmatrix}$ & $\begin{pmatrix}
u^{-1} & 0 \\ 
0 & u
\end{pmatrix}$ & $\begin{pmatrix}
0& u^{-1} \\ 
-u &0
\end{pmatrix}$  \\ 
& & & & \\
\end{tabular}
\end{center}  
\caption{Monodromy on a generic line.}
\label{tab:monodromyLinesQuint}
\end{table}

\begin{rema}
Note however that not all virtually abelian representations have finite mapping class group orbits. Indeed, this is even false in the four punctured case, as evidenced by Mazzocco's work on Picard's solutions of the Painlevé VI equation~\cite{MazzPic}.
\end{rema}

\subsection{Mapping class group orbits}\label{sec:MCGorbits}
The link between algebraic isomonodromic deformations of punctured sphere and finite orbits under the mapping class group action have been extensively studied in recent years~\cite{DubMazz, Cousin,LT}. In this paragraph, we describe the orbits associated with the isomonodromic deformations discussed earlier and show that only two of them are in fact distinct.

First let us fix some notations; the class of a representation $$\rho : \Fb_4 = \langle d_1, \ldots, d_4 \, | \, \emptyset \rangle \rightarrow SL_2(\CC)$$ in the $SL_2(\CC)$--character variety $\CV(0,5)$ of the five punctured sphere is fully determined by the following:
\begin{align*}
t_1&:=\Tr(\rho(d_1)), \, t_2:=\Tr(\rho(d_2)), \,  t_3:=\Tr(\rho(d_3)), \,  \\
t_4&:=\Tr(\rho(d_4)), \,  t_5 := \Tr(\rho(d_1d_2d_3d_4)), \, \\
r_1&:=\Tr(\rho(d_1d_2)), \,  r_2:=\Tr(\rho(d_1d_3)), \,  r_3:=\Tr(\rho(d_1d_4)), \, \\
 r_4&:=\Tr(\rho(d_2d_3)), \,  r_5:=\Tr(\rho(d_2d_4)), \,   r_6:=\Tr(\rho(d_3d_4))
\end{align*}
and
\begin{align*}
r_7&:=\Tr(\rho(d_1d_2d_3)), \, r_8:=\Tr(\rho(d_1d_2d_4)), \, r_9:=\Tr(\rho(d_1d_3d_4)), \,
r_{10}:=\Tr(\rho(d_2d_3d_4)) .
\end{align*}

We know from Cousin's work~\cite{Cousin} that there is a correspondence between algebraic isomonodromic deformations of the five punctured sphere and finite orbits under the action of the pure mapping class group $\PMCG(0,5)$ on the character variety $\CV(0,5)$. It is known that (see for example Section~9.3 in~\cite{FarMar}) one has the following isomorphism (for any $n \geq 3$):
\begin{displaymath}
\PMCG(0,n+1) \cong PB_n \slash Z(PB_n) \; ;
\end{displaymath} 
where $PB_n$ is the pure braid group on $n$ strand, i.e the kernel of the natural group homomorphism $B_n \twoheadrightarrow \mathfrak{S}_n$. As such, $\PMCG(0,n+1)$ is an index $n!$ subgroup of the complete mapping class group $\MCG(0,n+1)$.

\subsection{Orbits under the pure mapping class group}

In order to prove Theorem~\ref{thA:MCGOrbits}, we need to explicitly compute the orbits of the families of representations concerned. We have done so using a straightforward "brute--force" algorithm implemented in \verb+Maple+. We make use of the following facts:
\begin{itemize}
\item the braid group on four strands $$B_4= \langle \sigma_1, \sigma_{2}, \sigma_2 \, | \, [\sigma_1,\sigma_3] ,\, \sigma_1 \sigma_2\sigma_1 = \sigma_2 \sigma_1 \sigma_2, \,\sigma_3 \sigma_2\sigma_3 = \sigma_2 \sigma_3 \sigma_2 \rangle \;$$ acts on the character variety $\CV(0,5)$ as follows: the braid $\sigma_i$ sends the class of some representation $\rho : \Fb_4 \rightarrow SL_2(\CC)$ to that of the representation $\rho^{\sigma_i}$ defined by:
\begin{displaymath}
\rho^{\sigma_i} : d_j \mapsto \left\lbrace \begin{array}{l}
\rho(d_i)\rho(d_{i+1})\rho(d_i)^{-1} \text{ if } j=i \\ 
\rho(d_i) \text{ if } j=i+1\\ 
\rho(d_j) \text{ else }
\end{array} \right. \; .
\end{displaymath}
\item The pure braid group $PB_4$ is generated by the following braids, for $1 \leq i <j \leq 3$:
\begin{displaymath}
\sigma_{i,j} := (\sigma_{j} \ldots \sigma_{i+1}) \sigma_i(\sigma_{j} \ldots \sigma_{i+1})^{-1} \; ;
\end{displaymath}
\item the centre of the pure braid group on four strand is generated by the squared fundamental braid
\begin{displaymath}
\Delta^2 := (\sigma_3 \sigma_2 \sigma_1 \sigma_2 \sigma_1 \sigma_1)^2 \;.
\end{displaymath}
It is quite straightforward to see that $\Delta^2$ acts on a representation $\rho$ as the conjugacy by $\rho(d_1) \rho(d_2)\rho(d_3)\rho(d_4)$ and thus acts trivially on $\CV(0,5)$. Therefore the orbit of the class of $\rho$ under the action of $\PMCG(0,5)$ is the same as its orbit under that of $PB_4$.
\end{itemize}

We will now look in turn at the two conclusions of Theorem~\ref{thA:MCGOrbits} by computing the associated orbits and using this data to reach the desired conclusion.

It follows from the remarks we made at the beginning of the proof that one only needs to compute the orbit of the classes of the representations $[\rho_i]$ under the pure braid group on four strand; which we achieve by using simple "brute--force" algorithms (should the reader want to know exactly how this is achieved, these algorithms are written down in Appendix A of \cite{AThese}).

In the coordinates $(\underline{t}\, | \, \underline{r})$ described earlier, the orbits of the classes of the representations $\rho_1, \ldots, \rho_4$ under the action of $\PMCG(0,5)$ are as follows:\\

[$\rho_1$] \begin{enumerate}
\item $\left(\dfrac{1+v^2}{v}, \dfrac{1+u^2}{u}, 0, 0, -\dfrac{u^2 +v^2}{uv} \, \right| \left. \, \dfrac{1+u^2 v^2}{uv}, 0, 0, 0, 0, -\dfrac{1+u^4}{u^2}, 0, 0, -\dfrac{u^4+v^2}{u^2 v}, -\dfrac{1+u^2}{u}\right)$
\item $\left(\dfrac{1+v^2}{v}, \dfrac{1+u^2}{u}, 0, 0, -\dfrac{u ^2 +v^2}{uv}\, \right| \left. \ \dfrac{u^2 +v^2}{uv}, 0, 0, 0, 0, -\dfrac{u^4+v^4}{u^2v^2}, 0, 0, -\dfrac{u^4+v^2}{u^2 v}, -\dfrac{u^2+v^4}{uv^2}\right) $
\item $\left(\dfrac{1+v^2}{v}, \dfrac{1+u^2}{u}, 0, 0, -\dfrac{u^2 +v^2}{uv}\, \right| \left. \ \dfrac{u ^2 +v^2}{uv}, 0, 0, 0, 0, -2, 0, 0, -\dfrac{1+v^2}{v}, -\dfrac{1+u^2}{u}\right) $
\item $\left(\dfrac{1+v^2}{v}, \dfrac{1+u^2}{u}, 0, 0, \dfrac{u^2 +v^2}{uv}\, \right| \left. \ \dfrac{1+u^2 v^2}{uv}, 0, 0, 0, 0, -\dfrac{1+v^4}{v^2}, 0, 0, -\dfrac{1+v^2}{v}, -\dfrac{u^2+v^4}{uv^2}\right) $
\end{enumerate}

[$\rho_2$] \begin{enumerate}
\item $\left(0, \dfrac{1+v^2}{v}, \dfrac{1+u^2}{u}, \dfrac{1+v^2}{v}, 0\, \right| \left. \ 0, 0, 0, \dfrac{1+u ^2 v^2}{uv}, \dfrac{1+v^4}{ v^2}, \dfrac{1+u ^2 v^2}{uv}, 0, 0, 0, \dfrac{1+u^2v^4}{u v^2} \right)$
\item $\left(0, \dfrac{1+v^2}{v}, \dfrac{1+u^2}{u}, \dfrac{1+v^2}{v}, 0\, \right| \left. \ 0, 0, 0, \dfrac{1+u ^2 v^2}{uv}, 2, \dfrac{u^2+v^2}{u v}, 0, 0, 0, \dfrac{1+u^2}{u}\right)$
\item $\left(0, \dfrac{1+v^2}{v}, \dfrac{1+u^2}{u}, \dfrac{1+v^2}{v}, 0\, \right| \left. \ 0, 0, 0, \dfrac{u^2+v^2}{u v}, \dfrac{1+v^4}{ v^2}, \dfrac{u^2+v^2}{u v}, 0, 0, 0, \dfrac{u^2+v^4}{u v^2}\right)$
\item $\left(0, \dfrac{1+v^2}{v}, \dfrac{1+u^2}{u}, \dfrac{1+v^2}{v}, 0\, \right| \left. \ 0, 0, 0, \dfrac{u^2+v^2}{u v}, 2, \dfrac{1+u ^2 v^2}{uv}, 0, 0, 0, \dfrac{1+u^2}{u}\right)$
\end{enumerate}

[$\rho_3$] \begin{enumerate}
\item $\left( 0, 0, \dfrac{1+s^2}{s }, \dfrac{1+s^2}{s }, -\dfrac{1+s^2}{s }\, \right| \left. \ -\dfrac{1+s^2}{s }, 0, 0, 0, 0, 2, -\dfrac{1+s^4}{s^2 }, -2, 0, 0\right)$
\item $\left(0, 0, \dfrac{1+s^2}{s }, \dfrac{1+s^2}{s }, -\dfrac{1+s^2}{s }\, \right| \left. \ -\dfrac{1+s^2}{s }, 0, 0, 0, 0, \dfrac{1+s^4}{s^2 }, -2, -2, 0, 0\right)$
\item $\left(0, 0, \frac{1+s^2}{s }, \frac{1+s^2}{s }, -\frac{1+s^2}{s }\, \right| \left. \ -\frac{(1+s^2)(1-s^2+s^4)}{s^3 }, 0, 0, 0, 0, \frac{1+s^4}{s^2 }, -\frac{1+s^4}{s^2 }, -\frac{1+s^4}{s^2 }, 0, 0\right)$
\item $\left(0, 0, \dfrac{1+s^2}{s }, \dfrac{1+s^2}{s }, -\dfrac{1+s^2}{s }\, \right| \left. \ -\dfrac{1+s^2}{s }, 0, 0, 0, 0, 2, -2, -\dfrac{1+s^4}{s^2 }, 0, 0\right)$
\end{enumerate}

[$\rho_4$] \begin{enumerate}
\item $\left( 0, \dfrac{1+s^2}{s }, \dfrac{1+s^2}{s }, \dfrac{1+s^2}{s }, 0\, \right| \left. \ 0, 0, 0, \dfrac{1+s^4}{s^2 }, 2, 2, 0, 0, 0, \dfrac{1+s^2}{s }\right)$
\item $\left(0, \frac{1+s^2}{s }, \frac{1+s^2}{s }, \frac{1+s^2}{s }, 0\, \right| \left. \ 0, 0, 0, \frac{1+s^4}{s^2 }, \frac{1+s^4}{s^2 }, \frac{1+s^4}{s^2 }, 0, 0, 0, \frac{(1+s^2)(1-s^2+s^4)}{s^3 }\right)$
\item $\left(0, \dfrac{1+s^2}{s }, \dfrac{1+s^2}{s }, \dfrac{1+s^2}{s }, 0\, \right| \left. \ 0, 0, 0, 2, 2, \dfrac{1+s^4}{s^2 }, 0, 0, 0, \dfrac{1+s^2}{s }\right)$
\item $\left(0, \dfrac{1+s^2}{s }, \dfrac{1+s^2}{s }, \dfrac{1+s^2}{s }, 0\, \right| \left. \ 0, 0, 0, 2, \dfrac{1+s^4}{s^2 }, 2, 0, 0, 0, \dfrac{1+s^2}{s }\right)$.
\end{enumerate}

One can then see just by checking the first four coordinates $t_1, \ldots, t_4$ (corresponding to the traces of the matrices $\rho(d_1), \ldots, \rho(d_4)$) that these constitute four distinct parametrised families of finite orbits, thus proving the first point of Theorem~\ref{thA:MCGOrbits}.

\subsection{Extended orbits}

In order to show that these orbits are "truly different", we wish to know whether or not any two of them are contained in some orbit under the mapping class group action of $\MCG(0,5)$ over $\CV(0,5)$. First, we use the extended orbit computation algorithm showcased in Appendix~A of the author's PhD thesis \cite{AThese} to actually compute the orbits of the classes of $\rho_1, \ldots, \rho_4$ under the aforementioned action. Since these range in size from 40 to 240 elements we shall refer the reader to that particular appendix for the complete list, and we will restrict ourselves to giving the highlights.

The first thing to remark is that the orbit of $\rho_1$ (resp. $\rho_2, \rho_3, \rho_4$) is made up of $240$ (resp. $120,120,40$) elements. This means that we have at least two distinct mapping class group orbits here : that of $\rho_1$, that of $\rho_2$ since they both have different cardinalities and the same number of free parameters. It now remains to see whether or not the family of orbits given by $\rho_3$ and $\rho_4$ are distinct from these two.

Looking at the list, one remarks that the fiftieth element in the orbit of $\rho_2$ is equal to 
\begin{displaymath}
\left( 0, 0, \dfrac{1+v^2}{v}, \dfrac{1+v^2}{v}, \dfrac{1+u^2}{u}, 0, 0, 0, 0, 2, \dfrac{1+u^2v^2}{uv}, \dfrac{u^2+v^2}{uv}, 0, 0, \dfrac{1+u^2}{u} \right) \;.
\end{displaymath}
A quick computation shows that the above becomes, after the change of parameters $u \mapsto -s, \, v\mapsto s$
\begin{displaymath}
\left( 0, 0, \dfrac{1+s^2}{s }, \dfrac{1+s^2}{s }, -\dfrac{1+s^2}{s }, 0, 0, 0, 0, 2, -\dfrac{1+s^4}{s^2 }, -2, 0, 0, -\dfrac{1+s^2}{s }\right)
\end{displaymath}
which is actually the first point in the extended orbit of $\rho_3$. Therefore, these two orbits must be equal. Moreover, the one--hundred and fourth element in the orbit of $\rho_2$ is equal to 
\begin{displaymath}
\left( 0,\dfrac{1+u^2}{u},\dfrac{1+v^2}{v},\dfrac{1+v^2}{v},0,0,0,\dfrac{1+u^2v^2}{uv},\dfrac{u^2+v^2}{uv},2,0,0,0,\dfrac{1+u^2}{u},0 \right) \;.
\end{displaymath}
Here again, an adequate change of parameters (namely $u,v \mapsto s$ turns it into the first element in the orbit of $\rho_4$
$$\left( 0, \dfrac{1+u^2}{u }, \dfrac{1+u^2}{u }, \dfrac{1+u^2}{u }, 0, 0, 0, \dfrac{1+u^4}{u^2 }, 2, 2, 0, 0, 0, \dfrac{1+u^2}{u }, 0\right)$$
meaning that this one is just a special case of that of $\rho_2$.

\section{Explicit construction of the remaining Garnier solution}\label{sec:explicitConstruction}
The object of this paragraph is to give an explicit construction of a family of logarithmic flat connections on $\PP^2_\CC$ whose monodromy realises the second representation in Theorem~\ref{thA:quintics} and to describe the associated isomonodromic deformation of the five punctured sphere. In order to do so, we will follow broadly the method outlined in \cite{A2}; without going into as much detail since both constructions are quite similar. 
\subsection{Set--up}

Consider the projective plane quintic curve $\Qr^\prime \subset \PP^2_\CC$ defined (in homogeneous coordinates $[x:y:z]$) by the equation: 
\begin{displaymath}
y(y-z)z(x^2-yz) = 0 \; .
\end{displaymath}

\begin{figure}[!h]
\begin{center}
\begingroup%
  \makeatletter%
  \providecommand\color[2][]{%
    \errmessage{(Inkscape) Color is used for the text in Inkscape, but the package 'color.sty' is not loaded}%
    \renewcommand\color[2][]{}%
  }%
  \providecommand\transparent[1]{%
    \errmessage{(Inkscape) Transparency is used (non-zero) for the text in Inkscape, but the package 'transparent.sty' is not loaded}%
    \renewcommand\transparent[1]{}%
  }%
  \providecommand\rotatebox[2]{#2}%
  \ifx\svgwidth\undefined%
    \setlength{\unitlength}{317.86725407bp}%
    \ifx\svgscale\undefined%
      \relax%
    \else%
      \setlength{\unitlength}{\unitlength * \real{\svgscale}}%
    \fi%
  \else%
    \setlength{\unitlength}{\svgwidth}%
  \fi%
  \global\let\svgwidth\undefined%
  \global\let\svgscale\undefined%
  \makeatother%
  \begin{picture}(1,0.5192419)%
    \put(0,0){\includegraphics[width=\unitlength]{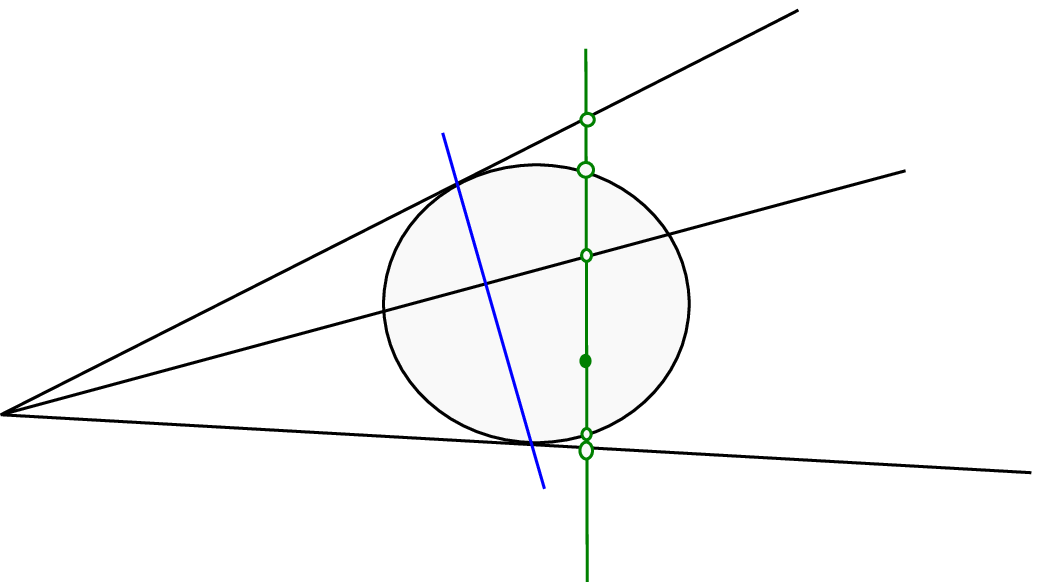}}%
    \put(0.525,0.20772693){\color[rgb]{0,0,0}\makebox(0,0)[lb]{\smash{$x_0$}}}%
      \put(0.57343331,0.10793112){\color[rgb]{0,0,0}\makebox(0,0)[lb]{\smash{$g_1$}}}%
    \put(0.56782522,0.16303938){\color[rgb]{0,0,0}\makebox(0,0)[lb]{\smash{$g_2$}}}%
    \put(0.567,0.29){\color[rgb]{0,0,0}\makebox(0,0)[lb]{\smash{$g_3$}}}%
    \put(0.57406327,0.36){\color[rgb]{0,0,0}\makebox(0,0)[lb]{\smash{$g_4$}}}%
    \put(0.53039702,0.455){\color[rgb]{0,0,0}\makebox(0,0)[lb]{\smash{$g_5$}}}%
    \put(0.32,0.41632942){\color[rgb]{0,0,0}\makebox(0,0)[lb]{\smash{$(x=0)$}}}%
    \put(0.73544196,0.33497499){\color[rgb]{0,0,0}\makebox(0,0)[lb]{\smash{$(y=z)$}}}%
        \put(0.73544196,0.08){\color[rgb]{0,0,0}\makebox(0,0)[lb]{\smash{$(y=0)$}}}%
    \put(0.69794266,0.49387034){\color[rgb]{0,0,0}\makebox(0,0)[lb]{\smash{$(z=0)$}}}%
  \end{picture}%
\endgroup
\end{center}
\caption{Fundamental group of $\PP^2 - \Qr^\prime$ and restriction to a generic line.} \label{fig:Pi1_Qp}
\end{figure}

We know from Degtyarev's list that the fondamental group of the complement of the quintic $\Qr^\prime$ is isomorphic to:
\begin{displaymath}
\Gamma_2^\prime := \langle a,b,c \, | \, [a,b] = [a,c^{-1}bc] = 1, \, (bc)^2 = (cb)^2 \rangle \; .
\end{displaymath}
More precisely, we get from the Zariski--Van Kampen method that we can take $c$ (resp. $a,b$) to be a loop realising the local monodromy around the conic $\Cr^\prime$ defined by the equation $x^2-yz=0$  (resp. the lines $(y=z)$, $(y=0)$). The Lefschetz hyperplane theorem (see~\cite{Milnor}, Theorem 7.4) tells us that the natural morphism $\tau : \Fb_4 \rightarrow \Gamma_2^\prime$ stemming from restriction to a generic line is onto and if we chose a line passing through the base point used to define $a,b$ and $c$ then $\tau$ is given (up to changing the generators $g_i$) by (see Figure~\ref{fig:Pi1_Qp}):
\begin{align*}
g_1 & \mapsto  c \\
g_2 & \mapsto b\\
g_3 & \mapsto a\\
g_4 & \mapsto b\\
g_5 & \mapsto (cbab)^{-1} = (cb^2 a)^{-1} \; .
\end{align*}

As per Theorem~\ref{thA:quintics}, we wish to construct a family of logarithmic flat connections over $\PP^2$ with polar locus equal to $\Qr^\prime$ and monodromy of the form:
\begin{displaymath}
\rho_{u,v} : a  \mapsto \begin{pmatrix}
u & 0 \\ 
0 & u^{-1}
\end{pmatrix}  , \quad 
b  \mapsto \begin{pmatrix}
v & 0 \\ 
0 & v^{-1}
\end{pmatrix}, \quad
c  \mapsto \begin{pmatrix}
0 & 1 \\ 
-1 & 0
\end{pmatrix} \; , \text{ for }  u,v \in \CC^* \; .
\end{displaymath}

\subsection{A suitable double cover}

As in our earlier work \cite{A2}, the key point of our construction will be to find a properly ramified double cover so that we are able to obtain a connection with dihedral monodromy. Here, since the local monodromy with projective order two arises along the line $\lbrace y=0 \rbrace$, we will need to have the aforementioned cover ramify there. Consider the following $2:1$ birational map:
\begin{align*}
\pi : \PP^1 \times \PP^1 & \rightarrow \PP^2\\
([u_0:u_1],[v_0,v_1]) & \mapsto [u_0 v_1^2  : u_1v_0^2  : u_1 v_1^2] \; .
\end{align*}
The map $\pi$ has indeterminacy locus equal to the point $\lbrace ([1:0],[1:0])\rbrace$ and ramifies over $\lbrace y=0 \rbrace \subset \PP^2_\CC$. Moreover, one also has (see Fig.~\ref{fig:ramif}): 
\begin{align*}
\pi^* \Cr^\prime &= \lbrace ([u_0:u_1],[v_0:v_1]) \in  \PP^1 \times \PP^1 \, | \, (u_0v_1 - u_1v_0)(u_0v_1 +u_1v_0)=0\rbrace \\
& =\lbrace u=v \rbrace \cup \lbrace u=-v \rbrace \subset \PP^1_u \times \PP^1_v \; ,
\end{align*}
\begin{align*}
\pi^*( \lbrace y=z \rbrace \cap \lbrace z \neq 0 \rbrace) &= \lbrace ([u_0:u_1],[v_0:v_1]) \in  \PP^1 \times \PP^1 \, | \, v_0^2 = v_1^2 \rbrace \\
& = \lbrace v=1 \rbrace \cup \lbrace v=-1 \rbrace \subset \PP^1_u \times \PP^1_v \; ,
\end{align*}
\begin{align*}
\pi^* (\lbrace y=0 \rbrace\cap \lbrace z \neq 0 \rbrace) &= \lbrace ([u_0:u_1],[v_0:v_1]) \in  \PP^1 \times \PP^1 \, | \, v_0 = 0 \rbrace \\
& = \lbrace v=0 \rbrace \subset \PP^1_u \times \PP^1_v \; 
\end{align*}
and
\begin{align*}
\pi^* \lbrace z=0 \rbrace &= \lbrace ([u_0:u_1],[v_0:v_1]) \in  \PP^1 \times \PP^1 \, | \, u_1v_1^2 = 0 \rbrace \\
& =\lbrace u=\infty \rbrace \cup \lbrace v=\infty \rbrace \subset \PP^1_u \times \PP^1_v \; .
\end{align*}

\begin{figure}[!h]
\begin{center}
\scalebox{0.7}{
\begingroup%
  \makeatletter%
  \providecommand\color[2][]{%
    \errmessage{(Inkscape) Color is used for the text in Inkscape, but the package 'color.sty' is not loaded}%
    \renewcommand\color[2][]{}%
  }%
  \providecommand\transparent[1]{%
    \errmessage{(Inkscape) Transparency is used (non-zero) for the text in Inkscape, but the package 'transparent.sty' is not loaded}%
    \renewcommand\transparent[1]{}%
  }%
  \providecommand\rotatebox[2]{#2}%
  \ifx\svgwidth\undefined%
    \setlength{\unitlength}{549.725bp}%
    \ifx\svgscale\undefined%
      \relax%
    \else%
      \setlength{\unitlength}{\unitlength * \real{\svgscale}}%
    \fi%
  \else%
    \setlength{\unitlength}{\svgwidth}%
  \fi%
  \global\let\svgwidth\undefined%
  \global\let\svgscale\undefined%
  \makeatother%
  \begin{picture}(1,0.27441641)%
    \put(0,0){\includegraphics[width=\unitlength]{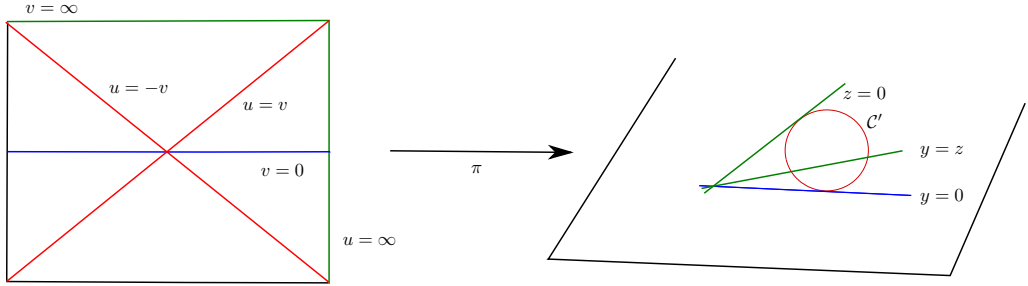}}%
    \put(0.33038598,0.0401361){\color[rgb]{0,0,0}\makebox(0,0)[lb]{\smash{$u=\infty$}}}%
    \put(0.01958122,0.26778241){\color[rgb]{0,0,0}\makebox(0,0)[lb]{\smash{$v=\infty$
}}}%
    \put(0.1017002,0.18878187){\color[rgb]{0,0,0}\makebox(0,0)[lb]{\smash{$u=-v$}}}%
    \put(0.23267478,0.17215017){\color[rgb]{0,0,0}\makebox(0,0)[lb]{\smash{$u=v$}}}%
    \put(0.25,0.108){\color[rgb]{0,0,0}\makebox(0,0)[lb]{\smash{$v=0$}}}%
    \put(0.45616314,0.11082081){\color[rgb]{0,0,0}\makebox(0,0)[lb]{\smash{$\pi$}}}%
    \put(0.84388945,0.154479){\color[rgb]{0,0,0}\makebox(0,0)[lb]{\smash{$\Cr^\prime$}}}%
    \put(0.82,0.18254499){\color[rgb]{0,0,0}\makebox(0,0)[lb]{\smash{$z=0$}}}%
    \put(0.8958635,0.12953145){\color[rgb]{0,0,0}\makebox(0,0)[lb]{\smash{$y=z$}}}%
    \put(0.8958635,0.08275482){\color[rgb]{0,0,0}\makebox(0,0)[lb]{\smash{$y=0$}}}%
  \end{picture}%
\endgroup}
\end{center}
\caption{Ramification locus of $\pi$.} \label{fig:ramif}
\end{figure}

This means that the quintic $\Qr^\prime$ is pulled back by $\pi$ onto seven lines in $\PP^1 \times \PP^1$. The cover $\pi$ corresponds to the quotient of $\PP^1 \times \PP^1$ under the involution $(u,v) \mapsto (u,-v)$. Now let us consider the elementary transformation of $\PP^1 \times \PP^1$ defined by:
\begin{align*}
\mathfrak{b} : \PP^1 \times \PP^1 & \rightarrow \PP^1 \times \PP^1\\
([u_0:u_1],[v_0:v_1]) & \mapsto ([u_0 v_0 :u_1v_1],[v_0:v_1]) \; ;
\end{align*}
then one can easily check that: 
\begin{align*}
\mathfrak{b}^* (\lbrace u = \pm v \rbrace\cap \lbrace v \neq 0,\infty \rbrace) &= \lbrace ([u_0:u_1],[v_0:v_1]) \in  \PP^1 \times \PP^1 \, | \, (u_0v_0 - \pm v_0)(u_1v_1 - \pm v1)=0\rbrace \\
& =\lbrace u= \pm 1 \rbrace \subset \PP^1_u \times \PP^1_v \; 
\end{align*}
and 
\begin{align*}
\mathfrak{b}^* \lbrace u = \infty \rbrace &= \lbrace ([u_0:u_1],[v_0,v_1]) \in  \PP^1 \times \PP^1 \, | \, u_1 v_1=0\rbrace \\
& =\lbrace u= \infty \rbrace \cup \lbrace v=\infty \rbrace \cup \lbrace v=\infty \rbrace\subset \PP^1_u \times \PP^1_v \; .
\end{align*}

The heuristic behind this procedure can be summed up as follows: imagine that we have a logarithmic flat connection $\nabla$ over $\PP^2$ satisfying the hypotheses of Theorem~\ref{thA:GarnierSol2}; then $\pi^* \nabla$ will be a logarithmic flat connection over $\PP^1 \times \PP^1$ whose monodromy around $\lbrace v = 0 \rbrace$ must be (projectively) trivial since $\pi$ ramifies over $\lbrace y=0 \rbrace$. As such, its monodromy group must be abelian and the pullback connection $(\pi \circ \mathfrak{b})^*\nabla$ also has abelian monodromy and factors through the fundamental group of the complement of $$\lbrace u=\infty \rbrace \cup \lbrace v=\infty \rbrace \cup \lbrace u=1 \rbrace \cup \lbrace u=-1 \rbrace \cup \lbrace v=1 \rbrace \cup \lbrace v=-1 \rbrace \text{ in } \PP^1_u \times \PP^1_v \; $$which is a product of free groups. This will allow us to set up a construction quite similar to the one undertaken in \cite{A2}.

\subsection{Constructing the connection}

Drawing inspiration from the above heuristic and our previous paper on this topic \cite{A2}, we consider the trivial rank two vector bundle $E_0$ on $X:= \PP^1 \times \PP^1$ endowed with the logarithmic flat connection:
\begin{displaymath}
\nabla_0 := \dd  + \dfrac{1}{2}\begin{pmatrix}
\omega_0 & 0 \\ 
0 & -\omega_0
\end{pmatrix} \; ,
\end{displaymath}
where  $u,v$ are projective coordinates on the base $X$ and
\begin{displaymath}
\omega_0 :=\lambda_0 \left( \dfrac{\dd u}{u-1} - \dfrac{\dd u}{u+1}  \right) + \lambda_1   \left( \dfrac{\dd v}{v-1} - \dfrac{\dd v}{v+1} \right)  \, ,
\end{displaymath}
with $(\lambda_0, \lambda_1) \in \CC^2\setminus \lbrace (0,0) \rbrace$. This connection generically has singular locus equal to four lines in $X$ and its local monodromy is given by the following matrices:
\begin{itemize}
\item around $\lbrace u = \pm 1\rbrace$:
\begin{displaymath}
\begin{pmatrix}
e^{- i \pi \lambda_0} & 0 \\ 
0 & e^{ i \pi \lambda_0}
\end{pmatrix}^{\pm 1} \; ; 
\end{displaymath}
\item around $\lbrace v =\pm 1\rbrace$, $j=0,1$:
\begin{displaymath}
\begin{pmatrix}
e^{ i \pi \lambda_j} & 0 \\ 
0 & e^{ i \pi \lambda_j}
\end{pmatrix}^{\pm 1} \; .
\end{displaymath}
\end{itemize}

One easily checks that the one--form $\omega_0$ is the pullback under the elementary transform $\mathfrak{b}$ of:
\begin{displaymath}
\omega_1 :=\lambda_0 \left( \dfrac{\dd u}{u-v} - \dfrac{\dd u}{u+v}    +\dfrac{u}{v}\left(\dfrac{\dd v}{u+v}- \dfrac{\dd v}{u-v}\right)\right) + \lambda_1   \left( \dfrac{\dd v}{v-1} - \dfrac{\dd v}{v+1}  \right)  \, .
\end{displaymath}
Note that this one--form gets non--trivial local monodromy around the lines $\lbrace v = 0 \rbrace$ and $\lbrace u = \infty \rbrace$; it naturally gives rise to a Riccati foliation, defined by the following one--form over $\PP(E_0) = X \times \PP^1$:
\begin{displaymath}
\Ric_1 := \dd w - \omega_1 w 
\end{displaymath}
where $w$ is a projective coordinate on the fibres.

Now define the following involution on the projective bundle $\PP(E_0)$:
 \begin{align*}
 \eta : X \times \PP^1 & \rightarrow X \times \PP^1\\
(u,v,w) & \mapsto (-u,-v,-w) \;
 \end{align*}
that leaves invariant the one--form $\Ric_1$. This means that if we extend $\pi$ into the map
\begin{align*}
\bar{\pi} : X \times \PP^1 & \rightarrow \PP^2 \times \PP^1 \\
(u,v,[w_0:w_1]) & \mapsto (\pi(u,v), [w_0+w_1:w_0-w_1])
\end{align*}
then $\Ric_1$ is the pullback under $\bar{\pi}$ of the Riccati one--form defined in some affine chart by:
\begin{align*}
\Ric := \dd w + & \left( \dfrac{\lambda_0}{x^2-y}\dd x + \left( \dfrac{\lambda_0 x}{x^2-y} - \dfrac{\lambda_1}{2(y-1)} \right) \dfrac{\dd y}{y} \right) w^2\\
-& \dfrac{\dd y }{2y} w 
-  \dfrac{\lambda_0 y}{x^2-y}\dd x + \dfrac{1}{2}\left( \dfrac{\lambda_0 x}{y(x^2-y)} - \dfrac{\lambda_1}{y-1}\right) \dd y \;.
\end{align*}
 
The above is a logarithmic one--form over  $\PP^2 \times \PP^1$ whose singular locus is exactly the quintic $\Qr^\prime$. Moreover, this lifts (as explained in details in \cite{A2}) to a logarithmic flat connection $\nabla = \nabla_{\lambda_0, \lambda_1}$ over the trivial bundle $\CC^2 \times \PP^2$ over the projective plane $\PP^2$, namely
\begin{displaymath}
\nabla = \dd + \begin{pmatrix}
 -\dfrac{\dd y}{4y} & -  \dfrac{\lambda_0 y}{x^2-y}\dd x + \dfrac{1}{2}\left( \dfrac{\lambda_0 x}{y(x^2-y)} - \dfrac{\lambda_1}{y-1}\right) \dd y\\ 
\dfrac{\lambda_0}{x^2-y}\dd x + \left( \dfrac{\lambda_0 x}{x^2-y} - \dfrac{\lambda_1}{2(y-1)} \right) \dfrac{\dd y}{y} &\dfrac{\dd y}{4y}
\end{pmatrix} \; ;
\end{displaymath}
see table~\ref{tab:resRic} for the exact residues.

\begin{table}
\begin{center}
\begin{tabular}{c|c|c}
Divisor & Residue & Eigenvalues \\ 
\hline 
& & \\
$y=0$ & $\begin{pmatrix}
- \dfrac{1}{4} & 0 \\ 
\dfrac{\lambda_0+ \lambda_1 x}{2 x} & \dfrac{1}{4}
\end{pmatrix} $ &$\pm \dfrac{1}{4}$ \\ 
& & \\
\hline 
& &\\
$y=1$ & $\begin{pmatrix}
0 &   - \dfrac{\lambda_1}{2}\\\ 
- \dfrac{\lambda_1}{2} &0
\end{pmatrix} $&$\pm\dfrac{\lambda_1}{2}$\\ 
& &\\
\hline 
& &\\
$\Cr$ & $\begin{pmatrix}
0 &    \dfrac{\lambda_0 x}{2}\\\ 
 \dfrac{\lambda_0 }{2x} &0
\end{pmatrix} $ &$\pm \dfrac{\lambda_0}{2}$ \\ 
& &\\
\hline 
& &\\
$L_\infty$  &  $\begin{pmatrix}
-\dfrac{1}{4} &   0\\\ 
 \dfrac{\lambda_1}{2} & \dfrac{1}{4}
\end{pmatrix} $ &$\pm \dfrac{1}{4}$\\ 
 & &
\end{tabular} 
\end{center}
\caption{Residues for $\nabla$.}
\label{tab:resRic}
\end{table}

The fact that this does not generically factor through an orbicurve can then be proven using exactly the same argument used to prove Theorem C in our previous paper \cite{A2}. Indeed, there generically is no algebraic relations between the coefficients of the matrices in table \ref{tab:resRic}.

\subsection{Restriction to generic lines}

We now fix $\lambda_0, \lambda_1 \in \CC^*$ and consider the connection induced by $\nabla$ on generic lines in $\PP^2$, such a line being given in the affine chart $\lbrace z=0 \rbrace$ by an equation of the form $y=a x+b$. As in our previous paper \cite{A2}, we obtain an isomonodromic deformation $(\nabla_{a,b})_{a,b}$ over the five punctured sphere whose associated Riccati one--form is given by the following local formula (in the aforementioned affine chart):
\begin{align*}
\Ric(\nabla_{a,b}) = & \dfrac{2a^2 y^4 + (4ab-2-2a^2)y^3+(2-4ab+2b^2)y^2-2b^2y}{2y(y-1)(a^2y^2+(2ab-1)y+b^2)}\dd w \\
& + \dfrac{\alpha_2(y)w^2 + \alpha_1(y)w + \alpha_0}{2y(y-1)(a^2y^2+(2ab-1)y+b^2)}\dd y
\end{align*}
where
\begin{align*}
\alpha_2(y) &:= (\lambda_0a+\lambda_1a^2)y^2+((a-b)\lambda_0+(2ab-1)\lambda_1)y+\lambda_0b+\lambda_1b^2 \\
 \alpha_1(y) &:=-a^2y^3+(a^2-2ab+1)y^2+(2ab-b^2-1)y+b^2  \\
\alpha_0(y) &:= -(\lambda_0a+\lambda_1a^2)y^3+((a+b)\lambda_0+(1-2ab)\lambda_1)y^2-(\lambda_0b+\lambda_1b^2)y \;.
\end{align*}

As such, if one chooses a parameter $c$ such that $c^2 = 1-4ab$ then one gets a family of logarithmic flat connections $(\nabla_{a,c})_{a,c}$ over $\PP^1 \setminus \lbrace 0,1,t_1,t_2,\infty \rbrace$, where:
\begin{displaymath}
t_1 = \left(\dfrac{c-1}{2a}\right)^2\quad \text{ and } \quad t_2 = \left(\dfrac{c+1}{2a}\right)^2\; .
\end{displaymath}
An explicit local expression for this isomonodromic deformation can be obtained simply by replacing $b$ by $\dfrac{1-c^2}{4a}$ in the above formula for $\Ric(\nabla_{a,b})$; this allows us to explicitly compute the spectral data for $\nabla_{a,c}$, as detailed in \cite{AThese}.

\subsection{Associated Garnier solution}
Start by setting
\begin{displaymath}
\hat{H}:=\dfrac{M_0}{y} + \dfrac{M_1}{y-1} + \dfrac{M_{2}}{y-t_1}+ \dfrac{M_{3}}{y-t_2} \;;
\end{displaymath}
then since the lower left coefficient of the residue at infinity $M_\infty$ of $\nabla_{a,c}$ is zero, the numerator of $\hat{H}$ must be a degree two polynomial in $y$, say:
\begin{equation} \label{eq:coeffSol2}
\hat{H}_{2,1}=\dfrac{c(t_1,t_2)(y^2-S_q(t_1,t_2)y+P_q(t_1,t_2))}{y(y-1)(y-t_1)(y-t_2)} \, ,
\end{equation}
where $S_q := q_1+q_2$ and $P_q := q_1 q_2$, with $q_1,q_2$ some algebraic functions of $(t_1,t_2)$. In the same manner that we did in earlier work on the subject \cite{A2}, we obtain a rational parametrisation of $S_q,P_q, t_1, t_2$.

More precisely, the parameters $(a,c)$ give a rational mapping $(\PP^1)^2 \rightarrow (\PP^1)^4$ giving explicit expressions of $(t_1,t_2,S_q,P_q)$, namely: 
\begin{align*}
t_1 & = \dfrac{(c-1)^2}{4} \, \\
t_2 & =  \dfrac{(c+1)^2}{4}\, , \\
S_q  & = \dfrac{(1-c^2+4a^2)\lambda_0 + 2a(c^2 + 1)\lambda_1 }{4a^2(\lambda_0+a\lambda_1)} \, ,\\
P_q  & = -\dfrac{(c-1)(c+1)(4a\lambda_0+(1-c^2))}{16a^3(\lambda_0+a\lambda_1)} \; .
\end{align*}

We can now prove that we have indeed constructed a family of algebraic solutions for a Garnier system. More precisely, it is a two--variables Hamiltonian system:
\begin{equation} \label{Garnier2}
\left\lbrace \begin{array}{ccc}
\dr_{t_k} \mathbf{p}_i & = - \dr_{\mathbf{q}_i} H_k  & i,k = 1,2\\ 
\dr_{t_k} \mathbf{q}_i &  = \dr_{\mathbf{p}_i} H_k & i,k = 1,2
\end{array} \right. \, ,
\end{equation} 
equivalent to the isomonodromy equation for five--punctured spheres, where the pair of Hamiltonians $(H_1,H_2)$ is explicitly known (the exact formulas however are a bit cumbersome; thus we refer the reader to  Appendix~B in \cite{AThese}; see also \cite{KimOka, Mazz}). Note that it differs from the one presented in \cite{A2}, as the local eigenvalues are not the same.

\begin{prop}\label{prop:Garnier2}
Let $q_1, q_2$ be the algebraic functions defined above; then there exist two algebraic functions $p_1(t_1,t_2)$ and $p_2(t_1,t_2)$ such that $(q_1,q_2,p_1,p_2)$ is a solution of (\ref{Garnier2}).
\end{prop} 
 \begin{proof}
We proceed exactly as we did for the first family of solutions: we consider the "symmetrised" system:
\begin{displaymath}
\left\lbrace \begin{array}{clc} 
\dr_{t_k} S_\mathbf{q} & = (\dr_{\mathbf{p}_1} + \dr_{\mathbf{p}_2})H_k & k=1,2 \\ 
\dr_{t_k} P_\mathbf{q} &  =   (\mathbf{q}_2\dr_{\mathbf{p}_1} + \mathbf{q}_1\dr_{\mathbf{p}_2})H_k & k=1,2 \\
\dr_{t_k} S_\mathbf{p} &  = - (\dr_{\mathbf{q}_1} + \dr_{\mathbf{q}_2})H_k& k=1,2 \\
\dr_{t_k} \gamma &  = \dfrac{-1}{(\mathbf{q}_1-\mathbf{q}_2)^2} ( (\mathbf{q}_1-\mathbf{q}_2)(\dr_{\mathbf{q}_1} + \dr_{\mathbf{q}_2})+(\mathbf{p}_1-\mathbf{p}_2)(\dr_{\mathbf{p}_1} + \dr_{\mathbf{p}_2}))H_k& k=1,2 \\
\end{array} \right. \, ,
\end{displaymath}
where $S_\mathbf{p} := \mathbf{p}_1+\mathbf{p}_2$ and $\gamma = \dfrac{\mathbf{p}_1-\mathbf{p}_2}{\mathbf{q}_1-\mathbf{q}_2}$. To obtain this we first had to consider the variable $\delta := \mathbf{q}_1-\mathbf{q}_2$ and then eliminate it using the fact that all expressions obtained had even degree in $\delta$ and that $\delta^2 = S_\mathbf{q}^2-4P_\mathbf{q}$.

Assume that $(p_1,p_2)$ are two algebraic functions such that $(q_1,q_2,p_1,p_2)$ is a solution of (\ref{Garnier2}). Using the first two equations with $k=1$ one then gets $S_\mathbf{p}$ and $\gamma$ as functions of $\dr_{t_1} S_q$ and $\dr_{t_1} P_q$ which in turn are rational functions of $(a,c)$, namely:
\begin{align*}
\gamma=&\dfrac{8(1+a)(\lambda_0+a\lambda_1)a^3}{(c^2-1)(1+c+2a)(1-c+2a)}\, , \\
Sp  = & -\dfrac{2a(1+3a-c^2+4a^2-3ac^2+4a^3)\lambda_0+(2a+4a^2+2ac^2)\lambda_1}{(c^2-1)(1+c+2a)(1-c+2a)}\: .
\end{align*}
This completes the rational parametrisation of all relevant variables and allows us to check that $(S_q,P_q,S_p,\gamma)$ indeed satisfies the above system, in the exact same way we did in \cite{A2}.
\end{proof}

\bibliographystyle{smfalpha}
\bibliography{2016-garnierQuint}

\end{document}